\documentclass[11pt]{article}
\usepackage{amssymb,amsthm,authblk}
\usepackage{amsmath,latexsym,epsf,mathabx}
\usepackage[sort&compress,numbers]{natbib} 
\usepackage{changepage}
\usepackage[all]{xy}
\usepackage{hyperref}
\numberwithin{equation}{section}
\usepackage{color}
\usepackage[mathlines,displaymath]{lineno}
\let\oldalign\align
\let\oldendalign\endalign

\renewenvironment{align}
{\linenomathNonumbers\oldalign}
{\oldendalign\endlinenomath}

\begin{document}
\theoremstyle{definition}
\newcommand{\nc}{\newcommand}
\def\PP#1#2#3{{\mathrm{Pres}}^{#1}_{#2}{#3}\setcounter{equation}{0}}
\def\mr#1{{{\mathrm{#1}}}\setcounter{equation}{0}}
\def\mc#1{{{\mathcal{#1}}}\setcounter{equation}{0}}
\def\mb#1{{{\mathbb{#1}}}\setcounter{equation}{0}}
\def\Mcc{\mc{C}}
\def\Mbe{\mb{E}}
\def\Mcp{\mc{P}}
\def\Mcg{\mc{G}}
\def\Mcw{\mc{W}}
\def\Mcx{\mc{X}}
\def\Mch{\mc{H}}
\def\extri{(\mc{C},\mb{E},\mathfrak{s})}
\def\GP{\mc{G}\mc{P}(\xi)}
\def\GI{\mc{G}\mc{I}(\xi)}
\def\P{\mc{P}(\xi)}
\def\Extri{\mb{E}\text{-triangle}}
\def\ext{\xi \text{xt}_{\xi}}
\def\Ext{\xi \text{xt}}
\def\res#1{\text{res\ }(#1)}
\def\>{\longrightarrow}
\def\Px{\mc{P_\Mcx}(\xi)}
\def\Gpx{\mc{GQP_{\Mcx}(\xi)}}
\def\Gx{\mathcal{G_{\Mcx}(\xi)}}
\newcommand{\resdim}[2]{{\text{res.dim}}_{\mc{#1}}{(#2)}}
\newcommand{\coresdim}[2]{\text{cores.dim}_{\mc{#1}(\xi)}#2}
\newtheorem{defn}{\bf Definition}[section]
\newtheorem{cor}[defn]{\bf Corollary}   
\newtheorem{prop}[defn]{\bf Proposition}
\newtheorem{thm}[defn]{\bf Theorem}
\newtheorem{lem}[defn]{\bf Lemma}
\newtheorem{rem}[defn]{\bf Remark}
\newtheorem{exam}[defn]{\bf Example}   
\newtheorem{fact}[defn]{\bf Fact}
\newtheorem{cond}[defn]{\bf Condition}
\def\bskip#1{{ \vskip 20pt }\setcounter{equation}{0}}
\def\sskip#1{{ \vskip 5pt }\setcounter{equation}{0}}
\def\mskip#1{{ \vskip 10pt }\setcounter{equation}{0}}
\def\bg#1{\begin{#1}\setcounter{equation}{0}}
\def\ed#1{\end{#1}\setcounter{equation}{0}}

\title{Quasi-resolving subcategories and dimensions in extriangulated categories}
\smallskip

\author[a]{\small Zhenggang He\thanks{Corresponding author:zhenggang\_he@163.com}}
\author[a]{\small Longfu Shi}
\author[b]{\small Shuangyan Li}
\affil[a]{\small Department of Public Foundation, Dali Nursing Vocational College, Dali 671006, China}
\affil[b]{\small Chuxiong Normal University, Chuxiong 675000,China}

\date{}
\maketitle
\baselineskip 15pt

\vskip 10pt%
\noindent {\bf Abstract}:
Let $\Mcc=(\Mcc,\Mbe,\mathfrak{s})$ be an extriangulated category with a proper class $\xi$ of $\Mbe$-triangles. In this paper, we introduce and study  quasi-resolving subcategories in $\Mcc$. More precisely, we first introduce the notion of $\Mcx$-resolution dimensions for a quasi-resolving subcategory $\Mcx$ of $\Mcc$ and then give some equivalent characterizations of objects which have finite $\Mcx$-resolution dimensions. As an application, we introduce Gorenstein quasi-resolving subcategories, denoted by $\Gpx$, in term of a quasi-resolving subcategory $\Mcx$, and prove that $\Gpx$ is also a quasi-resolving subcategory of $\Mcc$. Moreover, some classical known results are generalized in $\Gpx$.
 
\mskip\

\mskip\

\vskip 30pt

\section{Introduction}
\quad~ The study of resolving and quasi-resolving subcategories has played a pivotal role in the development of relative homological algebra and representation theory. Originating from the work of Auslander and Bridger on totally reflexive modules which are also called modules of Gorenstein dimension zero or finitely generated Gorenstein projective modules in \cite{MM}, resolving subcategories have been extensively investigated due to their applications in Gorenstein homological algebra and the theory of maximal Cohen-Macaulay modules. These subcategories, which are closed under extensions, kernels of epimorphisms, and direct summands, provide a unifying framework for understanding various homological properties across different algebraic contexts (see \cite{RT,ZhX1}). Zhu introduced the notion of quasi-resolving subcategories as a generalization of resolving subcategories in \cite{ZhX2}, which broadens the perspective by relaxing certain closure conditions while preserving essential homological features. Quasi-resolving subcategories have been shown to unify and extend many classical results, making them a powerful tool in modern homological algebra (see \cite{CW,ZhX2,ZY}).

The notion of extriangulated categories was introduced by Nakaoka and Palu in \cite{HY} as a simultaneous generalization of exact categories and triangulated categories.  Hence many results hold on exact categories and triangulated categories can be unified in the same framework. Moreover, there exist extriangulated categories which are neither exact nor triangulated, some examples can be found in \cite{HY,Zhou}. Let $(\Mcc,\Mbe, \mathfrak{s})$ be an extriangulated category. Hu, Zhang and Zhou \cite{JDP} studied a relative homological algebra in $\Mcc$ which parallels the relative homological algebra in a triangulated category. By specifying a class of $\Mbe$-triangles, which is called a proper class $\xi$ of $\Mbe$-triangles, they introduced $\xi$-$\mathcal{G}$projective dimensions and $\xi$-$\mathcal{G}$injective dimensions and discussed their properties. Also, they  gave some characterizations of $\xi$-$\mathcal{G}$injective dimension by using derived functors in \cite{JZP}.

This paper focuses on the homological theory of quasi-resolving subcategories in an extriangulated category $(\Mcc,\Mbe, \mathfrak{s})$. Our primary objective is to extend and unify existing results from abelian and triangulated categories to this more general setting. Specifically, we investigate the properties of quasi-resolving subcategories,  and their applications in studying resolution dimensions and finitistic dimensions.

The paper is organized as follows: In Section 2, we review the necessary preliminaries on extriangulated categories and proper classes of $\Mbe$-triangles. In Section 3, we  introduce the notion of quasi-resolving subcategories in extriangulated categories and establish their basic properties. Moreover, we deal with the syzygy objects of $\Mcx$-resolutions in a quasi-resolving subcategory $\Mcx$. In Section 4, we discuss the $\Mcx$-resolution dimensions relative to a quasi-resolving subcategory $\Mcx$, and provide several equivalent characterizations for objects with finite $\Mcx$-resolution dimension. In Section 5, As an application, we construct a new quasi-resolving subcategory $\Gpx$ which called Gorenstein  quasi-resolving subcategory  from a given quasi-resolving subcategory $\Mcx$, which generalizes the notion of $\xi$-Gorenstein projective objects given by Hu, Zhang and Zhou in \cite{JDP}. By applying the results from the previous sections to the Gorenstein quasi-resolving subcategory $\Gpx$, we have generalized some known conclusions.

Throughout this paper, all subcategories are assumed to be full and additive.

\section{Preliminaries}
\quad~Throughout this paper, we always assume that $\Mcc= (\Mcc,\Mbe,\mathfrak{s})$  is an  extriangulated category with  a proper class $\xi$ of $\Mbe$-triangles, which has enough $\xi$-projectives and $\xi$-injectives and  satisfies Condition (WIC).  Let us briefly recall some basic definitions of extriangulated categories. We omit some details here, but the reader can find them in \cite{CZZ, YH,HY,HYA,JDP,JZP}.

\begin{defn}\citep[Definition 2.1 and 2.3]{HY}
	Suppose that $\mathcal{C}$ is equipped with an  biadditive functor $\mathbb{E}:\mathcal{C}^{op}\times \mathcal{C}\rightarrow \mathbf{Ab}$. For any pair of objects $A, C$ in $\mathcal{C}$, an element $\delta \in\mathbb{E}(C,A)$ is called an $\mathbb{E}$-extension. For any $a\in \mathcal{C}(A, A')$ and $c\in \mathcal{C}(C', C)$, we have $\mathbb{E}$-extensions $\mathbb{E}(C, a)(\delta)\in \mathbb{E}(C, A') $ and $\mathbb{E}(c,A)(\delta)\in \mathbb{E}(C', A).$ We abbreviately denote them by $a_{*}\delta$ and $c^* \delta$ respectively. For any $ A$, $ C\in \mathcal{C}$, the zero element $0\in\mathbb{E}(C, A)$ is called the split $\mathbb{E}$-extension. And let $\delta\in \mathbb{E}(C,A)$ and $\delta '\in \mathbb{E}(C',A')$ be any pair of $\mathbb{E}$-extensions. A morphism $(a,c):\delta\rightarrow\delta'$ of $\mathbb{E}$-extensions is a pair of morphism $a\in \mathcal{C}(A,A')$ and $c\in \mathcal{C}(C,C')$ in $\mathcal{C}$ satisfying the equality
	\[
		a_*\delta=c^*\delta '.
	\]
\end{defn}

\begin{defn}\citep[Definition 2.7]{HY}
	Let $A$, $C\in \mathcal{C}$ be any pair of objects. Two sequences of morphisms $A\stackrel{x}\longrightarrow B\stackrel{y}\longrightarrow C$ and $A\stackrel{x'}\longrightarrow B'\stackrel{y'}\longrightarrow C$ in $\mathcal{C}$ are said to be equivalent if there exists an isomorphism $b\in \mathcal{C}(B,B')$ which makes the following diagram commutative.
	\[
	\xymatrix{
		A\ar@{=}[d]\ar[r]^x &B\ar[d]^b\ar[r]^y &C\ar@{=}[d]\\
		A\ar[r]^{x'} &B'\ar[r]^{y'} &C}
	\]
\end{defn}	

We denote the equivalence class of $A\stackrel{x}\longrightarrow B\stackrel{y}\longrightarrow C$ by [$A\stackrel{x}\longrightarrow B\stackrel{y}\longrightarrow C$].

\begin{defn}\citep[Definition 2.8]{HY}
	(1) For any $A$, $C\in \mathcal{C}$, we let
	\[
		0=[A \stackrel{\begin{tiny}\begin{bmatrix}
		1 \\
		0
		\end{bmatrix}\end{tiny}}{\longrightarrow} A \oplus C \stackrel{\begin{tiny}\begin{bmatrix}
		0&1
		\end{bmatrix}\end{tiny}}{\longrightarrow} C].
	\] 
	(2) For any $[A\stackrel{x}\longrightarrow B\stackrel{y}\longrightarrow C]$ and $[A'\stackrel{x'}\longrightarrow B'\stackrel{y'}\longrightarrow C']$, we let
	\[
		[A\stackrel{x}\longrightarrow B\stackrel{y}\longrightarrow C]\oplus[A'\stackrel{x'}\longrightarrow B'\stackrel{y'}\longrightarrow C']=[A\oplus A'\stackrel{x\oplus x'}\longrightarrow B\oplus B'\stackrel{y\oplus y'}\longrightarrow C\oplus C'].
	\] 
\end{defn}

\begin{defn}\citep[Definition 2.9]{HY}
	Let $\mathfrak{s}$ be a correspondence which associates an equivalence class $\mathfrak{s}(\delta)=[A\stackrel{x}{\longrightarrow}B\stackrel{y}{\longrightarrow}C]$ to any $\mathbb{E}$-extension $\delta\in\mathbb{E}(C,A)$ . This $\mathfrak{s}$ is called a  realization of  $\mathbb{E}$, if for any morphism $(a,c):\delta\rightarrow\delta'$ with $\mathfrak{s}(\delta)=[A\stackrel{x}\longrightarrow B\stackrel{y}\longrightarrow C]$ and $\mathfrak{s}(\delta')=[A'\stackrel{x'}\longrightarrow B'\stackrel{y'}\longrightarrow C']$, there exists $b\in \mathcal{C}$ which makes the following  diagram commutative.
	\[
		\xymatrix{
		& A \ar[d]^{a} \ar[r]^{x} & B  \ar[r]^{y}\ar[d]^{b} & C \ar[d]^{c}    \\
		&A'\ar[r]^{x'} & B' \ar[r]^{y'} & C'}
	\]
In the above situation, we say that the triplet $(a,b,c)$ realizes $(a,b)$.
\end{defn}

\begin{defn}\citep[Definition 2.10]{HY}
	Let $\mathcal{C},\mathbb{E}$ be as above. A realization $\mathfrak{s}$ of $\mathbb{E}$ is said to be {\em additive} if it satisfies the following conditions.
	
	(a) For any $A,~C\in\mathcal{C}$, the split $\mathbb{E}$-extension $0\in\mathbb{E}(C,A)$ satisfies $\mathfrak{s}(0)=0$.
	
	(b) $\mathfrak{s}(\delta\oplus\delta')=\mathfrak{s}(\delta)\oplus\mathfrak{s}(\delta')$ for any pair of $\mathbb{E}$-extensions $\delta$ and $\delta'$.
	
\end{defn}

\begin{defn}\citep[Definition 2.12]{HY}
	A triplet $(\mathcal{C}, \mathbb{E},\mathfrak{s})$ is called an extriangulated category if it satisfies the following conditions. \\
	$\rm(ET1)$ $\mathbb{E}$: $\mathcal{C}^{op}\times\mathcal{C}\rightarrow \mathbf{Ab}$ is a biadditive functor.\\
	$\rm(ET2)$ $\mathfrak{s}$ is an additive realization of $\mathbb{E}$.\\
	$\rm(ET3)$ Let $\delta\in\mathbb{E}(C,A)$ and $\delta'\in\mathbb{E}(C',A')$ be any pair of $\mathbb{E}$-extensions, realized as
	\[
		\mathfrak{s}(\delta)=[A\stackrel{x}{\longrightarrow}B\stackrel{y}{\longrightarrow}C] \quad \text{and} \quad \mathfrak{s}(\delta')=[A'\stackrel{x'}{\longrightarrow}B'\stackrel{y'}{\longrightarrow}C'].
	\]
For any pair $(a, b)$ defining such a commutative square in $\mathcal{C}$, 	
\[
	\xymatrix{
		A \ar[d]^{a} \ar[r]^{x} & B \ar[d]^{b} \ar[r]^{y} & C \\
		A'\ar[r]^{x'} &B'\ar[r]^{y'} & C'}
\]
there is  a morphism $c:C\to C'$  that the pair $(a,c)$ defines a morphism of extensions: $\delta\rightarrow\delta'$ which is realized by the triple $(a,b,c)$.\\
	$\rm(ET3)^{op}$ Dual of $\rm(ET3)$.\\
    $\rm(ET4)$ Let $\delta\in \mathbb{E}(D,A)$ and $\delta'\in \mathbb{E}(F,B)$ be $\mathbb{E}$-extensions respectively realized by
	\[
		A\stackrel{f}{\longrightarrow}B\stackrel{f'}{\longrightarrow}D\quad \text{and }\quad B\stackrel{g}{\longrightarrow}C\stackrel{g'}{\longrightarrow}F.
	\]
Then there exist an object $E\in\mathcal{C}$, a commutative diagram
	\[
		\xymatrix{
		A \ar@{=}[d]\ar[r]^{f} &B\ar[d]^{g} \ar[r]^{f'} & D\ar[d]^{d} \\
		A \ar[r]^{h} & C\ar[d]^{g'} \ar[r]^{h'} & E\ar[d]^{e} \\
		& F\ar@{=}[r] & F }
	\]
in $\mathcal{C}$, and an $\mathbb{E}$-extension $\delta''\in \mathbb{E}(E,A)$ realized by $A\stackrel{h}{\longrightarrow}C\stackrel{h'}{\longrightarrow}E$, which satisfy the following compatibilities.\\
	$(\textrm{i})$ $D\stackrel{d}{\longrightarrow}E\stackrel{e}{\longrightarrow}F$ realizes $f'_*\delta'$,\\
	$(\textrm{ii})$ $d^*\delta''=\delta$,\\
	$(\textrm{iii})$ $f_*\delta''=e^*\delta'$.\\
	$\rm(ET4)^{op}$ Dual of $\rm(ET4)$.
\end{defn}

\begin{rem}
	Both exact categories and triangulated categories can be viewed as extriangulated categories, as demonstrated in \citep[Example 2.13]{HY}. It should be noted that there exist extriangulated categories that are neither exact categories nor triangulated categories, with examples provided in \citep[Proposition 3.30]{HY} and \citep[Example 4.14]{Zhou}.
\end{rem}

We will use the following terminology.

\begin{defn}\citep[Definition 2.15 and 2.19]{HY}
	Let $(\mathcal{C},\mathbb{E},\mathfrak{s})$ be an extriangulated category.
	
	(1) A sequence $A\stackrel{x}\longrightarrow B\stackrel{y}\longrightarrow C$ is called conflation if it realizes some $\mathbb{E}$-extension $\delta\in \mathbb{E}(C,A)$. 
	
	(2) If a conflation $A\stackrel{x}\longrightarrow B\stackrel{y}\longrightarrow C$ realizes $\delta\in \mathbb{E}(C,A)$, we call the pair ($A\stackrel{x}\longrightarrow B\stackrel{y}\longrightarrow C, \delta$) an $\mathbb{E}$-triangle, and write it by
\[
 A\stackrel{x}\longrightarrow B \stackrel{y}\longrightarrow C  \stackrel{\delta}\dashrightarrow .
\]
We usually don't write this “ $\delta$” if it not used in the argument.
	
	(3) Let $A\stackrel{x}\longrightarrow B \stackrel{y}\longrightarrow C  \stackrel{\delta}\dashrightarrow $ and  $A'\stackrel{x'}\longrightarrow B' \stackrel{y'}\longrightarrow C'  \stackrel{\delta'}\dashrightarrow $ be any pair of $\mathbb{E}$-triangles. If a triplet $(a,b,c)$ realizes $(a,c):\delta \rightarrow\delta'$, then we write it as	
	\[
		\xymatrix{
			A \ar[d]^{a} \ar[r]^{x} & B  \ar[r]^{y}\ar[d]^{b} & C \ar[d]^{c} \ar@{-->}[r]^{\delta} &  \\
			A'\ar[r]^{x'} & B' \ar[r]^{y'} & C' \ar@{-->}[r]^{\delta'} &    }
	\]
and call $(a,b,c)$ a morphism of $\mathbb{E}$-triangles.
\end{defn}

	\begin{lem}\label{BH}\citep[Proposition 3.15]{HY}
		Let $(\mathcal{C},\mathbb{E},\mathfrak{s})$ be an extriangulated category. Then the following hold.
		
		(1) Let $C$  be any object, and let $A_1\stackrel{x_1}{\longrightarrow}B_1\stackrel{y_1}{\longrightarrow}C\stackrel{\delta_1}\dashrightarrow$  and $A_2\stackrel{x_2}{\longrightarrow}B_2\stackrel{y_2}{\longrightarrow}C\stackrel{\delta_2}\dashrightarrow$  be any pair of $\mathbb{E}$-triangles. Then there is a commutative diagram in $\mathcal{C}$
	\[
		\xymatrix{
			&{A_2}\ar[d]^{m_2}\ar@{=}[r]&{A_2}\ar[d]^{x_2}\\
			{A_1}\ar@{=}[d]\ar[r]^{m_1}&M\ar[r]^{e_1}\ar[d]^{e_2}&{B_2}\ar[d]^{y_2}\\
			 {A_1}\ar[r]^{x_1}&{B_1}\ar[r]^{y_1}&C}
	\]
		\noindent which satisfies $\mathfrak{s}(y^*_2\delta_1)=[A_1\stackrel{m_1}{\longrightarrow}M\stackrel{e_1}{\longrightarrow}B_2]$, $\mathfrak{s}(y^*_1\delta_2)=[A_2\stackrel{m_2} {\longrightarrow}M\stackrel{e_2}{\longrightarrow}B_1]$ and $m_{1*}\delta_1 +m_{2*}\delta_2 = 0$.
		
		(2)  Let $A$  be any object, and let $A\stackrel{x_1}{\longrightarrow}B_1\stackrel{y_1}{\longrightarrow}C_1\stackrel{\delta_1}\dashrightarrow$  and $A\stackrel{x_2}{\longrightarrow}B_2\stackrel{y_2}{\longrightarrow}C_2\stackrel{\delta_2}\dashrightarrow$  be any pair of $\mathbb{E}$-triangles. Then there is a commutative diagram in $\mathcal{C}$
		\[
			\xymatrix{
				A\ar[d]^{x_2}\ar[r]^{x_1}&{B_1}\ar[d]^{m_2}\ar[r]^{y_1}&{C_1}\ar@{=}[d]\\
				{B_2}\ar[d]^{y_2}\ar[r]^{m_1}&M\ar[r]^{e_1}\ar[d]^{e_2}&{C_1}\\
				{C_2}\ar@{=}[r]&{C_2}
			}
		\]
		\noindent which satisfies $\mathfrak{s}(x_{2_*}\delta_1)=[B_2\stackrel{m_1}{\longrightarrow}M\stackrel{e_1}{\longrightarrow}C_1]$, $\mathfrak{s}(x_{1_*}\delta_2)=[B_1\stackrel{m_2}{\longrightarrow}M\stackrel{e_2}{\longrightarrow}C_2]$ and $e_1^*\delta_1 +e_2^*\delta_2 = 0$.
	\end{lem}
Here we introduce the weak idempotent completeness condition (WIC) for an extriangulated categories.

\begin{cond}\citep[Condition 5.81]{HY} (WIC Condition)\label{WIC}
	Consider the following conditions.
	
	(1) Let $f\in\Mcc(A,B),g\in\Mcc(B,C)$ be any pair of morphisms. If $gf$ is an inflation, then so is $f$.
	
 	(2) Let $f\in\Mcc(A,B),g\in\Mcc(B,C)$ be any pair of morphisms. If $gf$ is a deflation, then so is $g$.
\end{cond}

Hu, Zhang and Zhou formally defined the following concepts in \citep{JDP}.

Let $\xi$ be a class of $\mathbb{E}$-triangles. One says $\xi$ is closed under base change if for any $\mathbb{E}$-triangle
	\[ 
	A\stackrel{x}\longrightarrow B \stackrel{y}\longrightarrow C  \stackrel{\delta}\dashrightarrow 
		\in \xi
	\]
and any morphism $c:C'\rightarrow C$, then any  $\mathbb{E}$-triangle $A\stackrel{x'}\longrightarrow B' \stackrel{y'}\longrightarrow C'  \stackrel{c^*\delta}\dashrightarrow $ belongs to $\xi$. Dually, one says $\xi$ is closed under cobase change if for any $\mathbb{E}$-triangle
	\[ 
	A\stackrel{x}\longrightarrow B \stackrel{y}\longrightarrow C  \stackrel{\delta}\dashrightarrow 
		\in \xi
	\]
and any morphism $a:A\rightarrow A'$, then any $\mathbb{E}$-triangle $A'\stackrel{x'}\longrightarrow B' \stackrel{y'}\longrightarrow C  \stackrel{a_*\delta}\dashrightarrow$ belongs to $\xi$.

A class of $\mathbb{E}$-triangles $\xi$ is called saturated if in the situation of Lemma \ref{BH}(1), whenever  $A_2\stackrel{x_2}\longrightarrow B_2 \stackrel{y_2}\longrightarrow C  \stackrel{\delta_2}\dashrightarrow$ and $A_1\stackrel{m_1}\longrightarrow M \stackrel{e_2}\longrightarrow B_2  \stackrel{y_2^*\delta_1}\dashrightarrow$ belong to $\xi$, then the
	$\mathbb{E}$-triangle $A_1\stackrel{x_1}\longrightarrow B_1 \stackrel{y_1}\longrightarrow C  \stackrel{\delta_1}\dashrightarrow$ belongs to  $\xi$. 

An $\mathbb{E}$-triangle $A\stackrel{x}\longrightarrow B \stackrel{y}\longrightarrow C  \stackrel{\delta}\dashrightarrow $ is called split if $\delta=0$. It is easy to see that it is split if and only if $x$ is section or $y$ is retraction. And we denote the subclass of $\mathbb{E}$-triangles  by $\Delta_0$ which consisting of the split $\mathbb{E}$-triangle. 
\begin{defn}\label{ZL}\citep[Definition 3.1]{JDP}
Let $\xi$ be a class of $\mathbb{E}$-triangles which is closed under isomorphisms. $\xi$ is called a proper class  of $\mathbb{E}$-triangles if the following conditions holds:
	
	(1) $\xi$ is closed under finite coproducts and $\Delta_0 \subseteq \xi$.
	
	(2) $\xi$ is closed under base change and cobase change.
	
	(3) $\xi$ is saturated.
\end{defn}

Following lemma will be used many times in this paper.
\begin{lem}\label{Thm3.2}Lemma \ref{Thm3.2}
 Let $\xi$ be a class of $\mathbb{E}$-triangles which is closed under isomorphisms.
Set $\mathbb{E}_\xi:=\mathbb{E}|_\xi$, that is, $$\mathbb{E}_\xi(C, A)=\{\delta\in\mathbb{E}(C, A)~|~\delta~ \textrm{is realized as an $\mathbb{E}$-triangle}\xymatrix{A\ar[r]^x&B\ar[r]^y&C\ar@{-->}[r]^{\delta}&}~\textrm{in}~\xi\}$$ for any $A, C\in\mathcal{C}$, and $\mathfrak{s}_\xi:=\mathfrak{s}|_{\mathbb{E}_\xi}$. Then $\xi$ is a  proper class  of $\mathbb{E}$-triangles if and only if $(\mathcal{C}, \mathbb{E}_\xi, \mathfrak{s}_\xi)$ is an extriangulated category.
\end{lem}

Recall that a morphism $x:A\to B$ is called ($\xi$-)inflation, if there exists an $\Mbe$-triangle $A\stackrel{x}\longrightarrow B\longrightarrow C\dashrightarrow $ (in $\xi$). And  a morphism $y: B\to C $ is called ($\xi$-)deflation,  if there exists an $\Mbe$-triangle $A\longrightarrow B\stackrel{y}\longrightarrow C\dashrightarrow$ (in $\xi$).

\begin{defn}\citep[Definition 4.1]{JDP} An object $P\in \mathcal{C}$ is called $\xi$-projective if for any $\mathbb{E}$-triangle
	\[
		A\stackrel{x}\longrightarrow B \stackrel{y}\longrightarrow C  \stackrel{\delta}\dashrightarrow 
	\]
	lies in $\xi$, the induced sequence of abelian groups
	\[
	0\longrightarrow\mathcal{C}(P,A)\longrightarrow \mathcal{C}(P,B)\longrightarrow\mathcal{C}(P,C)\longrightarrow 0
	\]
is exact. Dually, we have the definition of $\xi$-injective.
\end{defn}

We denote $\P$ (resp. $\mathcal{I}(\xi)$) the class of $\xi$-projective (resp. $\xi$-injective) objects of $\Mcc$. An extriangulated category $(\mathcal{C},\mathbb{E},\mathfrak{s})$ is said to have enough $\xi$-projectives  (resp. enough $\xi$-injectives ) provided that for each object $A\in \mathcal{C}$ there exists an $\mathbb{E}$-triangle $K\longrightarrow P\longrightarrow A\dashrightarrow$ (resp. $A\longrightarrow I\longrightarrow K\dashrightarrow$) in $\xi$ with $P\in\mathcal{P}(\xi)$(resp. $I\in\mathcal{I}(\xi)$). One can also find the above concept in \citep[Section 4]{JDP}.

\begin{defn}\citep[Definition 4.4]{JDP}
	An unbounded complex $\mathbf{X}$ is called $\xi$-exact if   $\mathbf{X}$ is a diagram
	\[
		\cdots \longrightarrow X_1\xrightarrow {d_1}X_0 \xrightarrow{d_0}X_{-1}\longrightarrow\cdots
	\]
in $\mathcal{C}$ such that for each integer $n$, there exists an $\mathbb{E}$-triangle $K_{n+1}\stackrel{g_n}\longrightarrow X_n\stackrel{f_n}\longrightarrow K_n\stackrel{\delta_n}\dashrightarrow$ in $\xi$ and $d_n=g_{n-1}f_n$. These $\mathbb{E}$-triangles are called the $\xi$-resolution $\mathbb{E}$-triangles of the $\xi$-exact complex $\mathbf{X}$.
\end{defn}

\begin{defn}\citep[Definition 4.5 and 4.6]{JDP}
Let $\mathcal{W}$ be a class of objects in $\Mcc$. An $\mathbb{E}$-triangle $A\longrightarrow B\longrightarrow C\dashrightarrow$ in $\xi$ is called to be $\Mcc(-,\mathcal{W})$-exact (respectively $\Mcc(\mathcal{W},-)$-exact) if for any $W\in\mathcal{W}$, the induced sequence of abelian group $0 \rightarrow \Mcc(C,W)\rightarrow\Mcc(B,W)\rightarrow\Mcc(A,W)\rightarrow0$ (respectively $0 \rightarrow \Mcc(W,A)\rightarrow\Mcc(W,B)\rightarrow\Mcc(W,C)\rightarrow0$) is exact in $\mathbf{Ab}$.

A complex $\mathbf{X}$ is called $\Mcc(-,\mathcal{W})$-exact  (respectively $\Mcc(\mathcal{W},-)$-exact) if it is a $\xi$-exact complex with $\Mcc(-,\mathcal{W})$-exact $\xi$-resolution $\mathbb{E}$-triangles (respectively $\Mcc(\mathcal{W},-)$-exact $\xi$-resolution $\mathbb{E}$-triangles).

\end{defn}
\begin{defn}
	A $\xi$-projective resolution of an object $A\in\Mcc$ is a  right bounded and  $\xi$-exact complex
	\[
		\cdots\longrightarrow P_{n}\longrightarrow P_{n-1}\longrightarrow \cdots\longrightarrow P_{1}\longrightarrow P_{0}\longrightarrow A\longrightarrow 0
	\]
in $\Mcc$  with $P_{n}\in\P$ for all $n\geqslant0$. Dually, we can define the $\xi$-injective coresolution.
\end{defn}
 
 \begin{defn}\citep[Definition 3.2]{JZP}
	Let $A$ and $B$ be objects in $\mathcal{C}$.
	
	(1) If we choose a $\xi$-projective resolution $\mathbf{P}\longrightarrow A$ of $A$, then for any  integer $n\geqslant 0$, the $\xi$-cohomology groups $\xi \text{xt}_{\mathcal{P}(\xi)}^n(A, B)$ are defined as
	\[
	\xi \text{xt}_{\mathcal{P}(\xi)}^n(A, B)=H^n(\mathcal{C}(\mathbf{P},B)).
	\]
	
	(2) If we choose a $\xi$-injective coresolution $B\longrightarrow\mathbf{I}$ of $B$, then for any  integer $n\geqslant 0$, the $\xi$-cohomology groups $\xi \text{xt}_{\mathcal{I}(\xi)}^n(A, B)$ are defined as
	\[
	\xi \text{xt}_{\mathcal{I}(\xi)}^n(A, B)=H^n(\mathcal{C}(A,\mathbf{I})).
	\]
Then there exists an isomorphism $\xi \text{xt}_{\mathcal{P}(\xi)}^n(A, B)\backsimeq \xi \text{xt}_{\mathcal{I}(\xi)}^n(A, B)$, and denoting the isomorphism class of this abelian group by $\xi \text{xt}_{\xi}^n(A,B)$.
	\end{defn}

\begin{lem}\label{LZHL}\citep[Lemma 3.4 and its comment]{JZP}
	If $A\stackrel{x}\longrightarrow B \stackrel{y}\longrightarrow C  \stackrel{\delta}\dashrightarrow$ is an $\mathbb{E}$-triangle in $\xi$, then for any objects $X$  in $\mathcal{C}$, we have the following long exact sequences in $\mathbf{Ab}$
	\[
	0\longrightarrow \ext^0(X,A)\longrightarrow\ext^0(X,B)\longrightarrow\ext^0(X,C)\longrightarrow\ext^1(X,A)\longrightarrow\cdots
	\]
	and
	\[
	0\longrightarrow\ext^0(C,X)\longrightarrow\ext^0(B,X)\longrightarrow\ext^0(A,X)\longrightarrow \ext^1(C,X)\longrightarrow\cdots
	\]
For any objects $A$ and $B$, there is always a natural map $\varphi:\mathcal{C}(A,B)\rightarrow\ext^0(A,B)$, which is an isomorphism if $A\in \P$ or $B\in \mathcal{I}(\xi)$.
\end{lem}

	Let $\Mcx$ be a subcategory of $\Mcc$. We define
\[
	\begin{aligned}
	& \mathcal{X}^{\perp}:=\left\{C \in \mathcal{C} \mid \ext^{n\geq 1}(X, C)=0 \text { for all } X \in \mathcal{X}\right\} \\
	& ^{\perp} \mathcal{X}:=\left\{C \in \mathcal{C} \mid \ext^{n\geq 1}(C, X)=0 \text { for all } X \in \mathcal{X} \right\}
	\end{aligned}
\]
For two subcategories $\Mcx$ and $\mathcal{Y}$ of $\Mcc$, we say $\Mcx\perp\mathcal{Y}$ if $\Mcx\subseteq {}^\perp\mathcal{Y}$ ( equivalently $\mathcal{Y}\subseteq \Mcx^\perp$). 

\begin{defn}
	Let $\Mcx$ be a subcategory of $\Mcc$. An $\Mcx$-resolution of an object $M$ is a $\xi$-exact complex in $\Mcc$
\[
\mathbf{X}:\cdots \>X_2\>X_1\>X_0\>M\>0
\]
where $X_n$ is in $\Mcx$ for each integer $n$. In this case, for any integer $n\geq 1$, there exists an $\Mbe$-triangle 
\[
Kn\>X_{n-1}\>K_{n-1}\dashrightarrow (\text{set\ } K_0=M)
\]
in $\xi$ which is the $\xi$-resolution $\Mbe$-triangle of $\mathbf{X}$. Then the object $K_n$ is called an $n$th $\xi$-$\Mcx$-syzygy of $M$, denoted by $\Omega_{\Mcx}^n(M)$.  An object $M$ of $\Mcc$ is said to have $\Mcx$-resolution dimension $\leq n$, denoted by $\resdim{X}{M}\leq n$, if there is an $\Mcx$-resolution of the form
\[
0\>X_n\>X_{n-1}\>\cdots \>X_1\>X_0\>M\>0
\]
of $M$. If $n$ is the least such number, then we set $\resdim{X}{M}=n$ and if there is no such $n$, we set $\resdim{X}{M}=\infty$. And we define
\[
 \res{\Mcx}:=\left\{M \in \mathcal{C} \mid \text{\ $M$ has an $\Mcx$-resolution}\right\} \text{\ and\ } \widehat{\Mcx}:=\left\{M \in \mathcal{C} \mid \resdim{X}{M}<\infty \right\}.
\]
\end{defn}

\begin{defn}
Let $\mathcal{H}$ and $\Mcx$ be two subcategories of $\Mcc$ with $\mathcal{H}\subseteq \Mcx$. We say $\mathcal{H}$ is a $\xi$-cogenerator of $\Mcx$ if for any object $X\in \Mcx$, there is an $\Mbe$-triangle
\[
	X\> H \> X'\dashrightarrow 
\]
in $\xi$ with $H\in\mathcal{H}$ and $X'\in\Mcx$. Moreover, a $\xi$-cogenerator $\mathcal{H}$ is called $\Ext$-injective if $\Mcx\perp\mathcal{H}$ and $\ext^0(H,H')\cong \Mcc(H,H')$ for any $H,H'\in\Mch$.
\end{defn}

\begin{defn}
Assume that $\mathcal{X}$ is a subcategory of $\Mcc$, and let $A\stackrel{x}\longrightarrow B \stackrel{y}\longrightarrow C \stackrel{\delta}\dashrightarrow $ be any $\Mbe$-triangle in $\xi$. We call that

(1) $A$ the $\xi$-cocone of $y:B \to C$, and  call $C$ the cone of $x:A\to B$.

(2) $\Mcx$ is closed under $\xi$-extensions, if $A$ and $C$ are both in $\xi$, it holds that $B\in \Mcx$.

(3) $\Mcx$ is closed under $\xi$-cocones (resp. $\xi$-cones), if $B$ and $C$ lie in $\Mcx$ (resp. $A$ and $B$ lie in $\Mcx$), it holds that $A\in\Mcx$ (resp. $C\in\Mcx$).
\end{defn}

\section{Quasi-resolving subcategories}
\quad~In this section, we introduce and investigate the notion of quasi-resolving subcategories for an extriangulated category $\Mcc$ endowed with a proper class $\xi$ of $\Mbe$-triangles. We establish fundamental properties of the  $n$th $\xi$-$\Mcx$-syzygy of an object in a quasi-resolving subcategory $\Mcx$.

\begin{defn}
Let $\mathcal{X}$ be a subcategory of an extriangulated category $\Mcc$ and $\Px=\Mcx\cap\P$. Then $\Mcx$ is called quasi-resolving if it satisfies the following conditions.

(1)  $\Mcx$ is closed under $\xi$-extensions.

(2) $\Mcx$ is closed under $\xi$-cocones.

(3) $\Mcx\subseteq \res{\Px}$.   
\end{defn}

\begin{exam} 
Let $P\in\P$. Denote
\[
\mathcal{SF}(P):=\left\{ X\in\Mcc \mid  X \oplus P^m \cong P^n \text{\ for some\ } m,n\in\mathbb{N}^+\right\} 
\]
Then it is easy to see $\mathcal{SF}(P)$ is a quasi-resolving subcategory of $\Mcc$. Moreover, the  subcategory  
\[
P^\mathbb{N}:=\left\{ P^n \mid n\in\mathbb{N}^+\right\}
\] 
is a $\Ext$-injective cogenerator of $\mathcal{SF}(P)$.
\end{exam}

\begin{prop}\label{ZGQ}
Let $\Mcx$ be a quasi-resolving subcategory of $\Mcc$, then $\res{\Mcx}=\res{\Px}$.
\end{prop}
\begin{proof}
Obviously, $\res{\Px}\subseteq\res{\Mcx}$. Conversely, for any $M\in\res{\Mcx}$, there exists an $\Mcx$-resolution of $M$ 
\[
\cdots \>X_2\>X_1\>X_0\>M\>0
\]
with $\xi$-resolution $\Mbe$-triangle $Kn\>X_{n-1}\>K_{n-1}\dashrightarrow (\text{set\ } K_0=M)$ in $\xi$ for all $n\geq 1$. Since $\Mcx$ is  quasi-resolving, we have an $\Mbe$-triangle $L_1\>P_0\>X_0\dashrightarrow$ in $\xi$ with $P_0\in\Px$ and $L_1\in\Mcx$. It follows from Lemma \ref{Thm3.2} and $\rm(ET4)^{op}$ that we have the following commutative diagram 
\[
		\xymatrix{
			L_1\ar@{=}[d] \ar[r] & W_1 \ar[r]\ar[d] & K_1 \ar[d] \ar@{-->}[r] &  \\
			L_1\ar[r] & P_0\ar[d] \ar[r] & X_0\ar[d] \ar@{-->}[r]& \\
			& M\ar@{=}[r]\ar@{-->}[d]&M\ar@{-->}[d] \\
			& & &  }
\] 
where all  rows and columns are  $\mathbb{E}$-triangles in $\xi$. Hence there exists a commutative diagram 
\[
	\xymatrix{
		&K_2\ar@{=}[r]\ar[d]&K_2\ar[d]\\
		L_1\ar[r]\ar@{=}[d]&Y_1\ar[r]\ar[d]&X_1\ar[d]\ar@{-->}[r]&\\
		L_1\ar[r]&W_1\ar[r]\ar@{-->}[d]&K_1\ar@{-->}[r]\ar@{-->}[d]&\\
		&&&
	}
\]	
made of $\mathbb{E}$-triangles in $\xi$ by Lemma \ref{BH} (1) and Lemma \ref{Thm3.2}. Note that $L_1$ and $X_1$ are both in $\Mcx$, then we have $Y_1\in\Mcx$ since $\Mcx$ is closed under $\xi$-extensions. It implies that $W_1\in\res{\Mcx}$ since $K_2\in\res{\Mcx}$. Therefore, we obtain an $\Mbe$-triangle $W_1\>P_0\>M\dashrightarrow$ in $\xi$ with $P_0\in\Px$ and $W_1\in\res{\Mcx}$. Proceeding in this manner, one can get a $\xi$-exact complex
\[
\cdots \>P_2\>P_1\>P_0\>M\>0
\]
with $P_n\in\Px$ for all $n\geq 0$, which implies that $M\in\res{\Px}$. Thus $\res{\Mcx}\subseteq\res{\Px}$. So we have $\res{\Mcx}=\res{\Px}$.
\end{proof}

\begin{cor}
Let $\Mcx$ be a quasi-resolving subcategory of $\Mcc$, then $\widehat{\Mcx}\subseteq \res{\Px}$.
\end{cor}

\begin{lem}\label{lem3.5}
Let $\Mcx$ be a quasi-resolving subcategory of $\Mcc$, and let 
\[
0\>N\>X_1\>X_0\>M\>0
\]
is a $\xi$-exact complex with $X_0, X_1\in\Mcx$. Then there is a $\xi$-exact complex 
\[
0\>N\>X\>P\>M\>0
\]
with $X\in\Mcx$ and $P\in\Px$.	
\end{lem}
\begin{proof}
By the $\xi$-exact complex $N\>X_1\>X_0\>M$, one can get the following $\Mbe$-triangles
\[
N\>X_1\>K\dashrightarrow \text{\ and \ } K\>X_0\>M\dashrightarrow
\] 
 in $\xi$. Since $\Mcx$ is quasi-resolving, there is an $\Mbe$-triangle $L\>P\>X_0\dashrightarrow$ with $P\in\Px$ and $L\in\Mcx$. By Lemma \ref{Thm3.2} and $\rm(ET4)^{op}$, we have the following commutative diagram 
\[
		\xymatrix{
			L\ar@{=}[d] \ar[r] & Y \ar[r]\ar[d] & K \ar[d] \ar@{-->}[r] &  \\
			L\ar[r] & P\ar[d] \ar[r] & X_0\ar[d] \ar@{-->}[r]& \\
			& M\ar@{=}[r]\ar@{-->}[d]&M\ar@{-->}[d] \\
			& & &  }
\] 
where all  rows and columns are  $\mathbb{E}$-triangles in $\xi$. Moreover, it follows from Lemma \ref{Thm3.2} and  Lemma \ref{BH} (1) that we have the following commutative diagram
\[
	\xymatrix{
		&N\ar@{=}[r]\ar[d]&N\ar[d]\\
		L\ar[r]\ar@{=}[d]&X\ar[r]\ar[d]&X_1\ar[d]\ar@{-->}[r]&\\
		L\ar[r]&Y\ar[r]\ar@{-->}[d]&K\ar@{-->}[r]\ar@{-->}[d]&\\
		&&&
	}
\]	
made of $\mathbb{E}$-triangles in $\xi$. Connecting the middle columns in the above diagrams, we obtain the $\xi$-exact complex $N\>X\>P\>M$ with $X\in\Mcx$ and $P\in\Px$.
\end{proof}

\begin{thm}\label{thm3.6}
Let $\Mcx$ be a quasi-resolving subcategory and $n$ be a nonnegative integer. Assume that 
\[
0\>N\>X_{n-1}\>X_{n-2}\>\cdots\>X_1\>X_0\>M\>0\tag{$\spadesuit$}\label{HT}
\]
is a $\xi$-exact complex with $X_i\in\Mcx$ for all $0\leq i\leq n-1$. Then the following statemants hold.

(1) There exists a $\xi$-exact complex
\[
	0\>N\>X\>P_{n-2}\>\cdots\>P_1\>P_0\>M\>0
\]
with $X\in\Mcx$ and $P_i\in\Px$ for all $0\leq i\leq n-2$.

(2) There exists a $\xi$-exact complex
\[
	0\>L\>P_{n-1}\>P_{n-2}\>\cdots\>P_1\>P_0\>M\>0
\]
and an $\Mbe$-triangle $X\>L\>N\dashrightarrow$ in $\xi$ with $X\in\Mcx$ and $P_i\in\Px$ for all $0\leq i\leq n-1$.

(3) If $M\in\Mcx$, then there exists a $\xi$-exact complex
\[
	0\>N\>X\>P_{n-2}\>\cdots\>P_2\>P_1\>P_0\>0
\]
with $X\in\Mcx$ and $P_i\in\Px$ for all $0\leq i\leq n-2$.

(4) If $M\in\Mcx$, then there exists a $\xi$-exact complex
\[
	0\>T\>P_n\>P_{n-1}\>\cdots\>P_2\>P_1\>P_0\>0
\]
and an $\Mbe$-triangle $X\>T\>N\dashrightarrow$ in $\xi$ with $X\in\Mcx$ and $P_i\in\Px$ for all $0\leq i\leq n$.

(5) Let $\Mch$ be a $\xi$-cogenerator of $\Mcc$, then for any integer $m$ with $0\leq m\leq n-1$, there exists a $\xi$-exact complex 
\[
	0\>N\>H_{n-1}\>\cdots\>H_{m+1}\>X\>P_{m-1}\>\cdots\>P_0\>M\>0
\]
with $X\in\Mcx$, $H_i\in\Mch$ for any $m+1\leq i\leq n-1$ and $P_j\in\Px$ for any  $0\leq j\leq m-1$.

(6) Let $\Mch$ be a $\xi$-cogenerator of $\Mcc$ and $N\in\Mcx$, then for any integer $m$ with $0\leq m\leq n-1$, there exists a $\xi$-exact complex 
\[
	0\>H_n\>H_{n-1}\>\cdots\>H_{m+1}\>X\>P_{m-1}\>\cdots\>P_0\>M\>0
\]
with $X\in\Mcx$, $H_i\in\Mch$ for any $m+1\leq i\leq n$ and $P_j\in\Px$ for any  $0\leq j\leq m-1$.
\end{thm}
\begin{proof}
(1)  We will prove the result by induction on $n$. When $n=0$, the result is true obviously. When $n=1$, the result just follows from Lemma \ref{lem3.5}. Now we assume that $n\geq 2$. Since (\ref{HT}) is a $\xi$-exact complex, there exist an $\Mbe$-triangle $N\>X_{n-1}\>K\dashrightarrow$ in $\xi$ and a $\xi$-exact complex
\[
	0\>K\>X_{n-2}\>X_{n-3}\>\cdots\>X_1\>X_0\>M\>0
\]
in $\Mcc$. By induction hypothesis, there is a $\xi$-exact complex 
\[
	0\>K\>X'\>P_{n-3}\>\cdots\>P_1\>P_0\>M\>0
\]
with $X'\in\Mcx$ and $P_i\in\Px$ for any $0\leq i\leq n-3$. Then we obtain a new $\xi$-exact complex
\[
	0\>N\>X_{n-1}\>X'\>P_{n-3}\>\cdots\>P_1\>P_0\>M\>0
\]
in $\Mcc$. Then there are two $\xi$-exact complex 
\[
	0\>N\>X_{n-1}\>X'\>L\>0
\]
and 
\[
	0\>L\>P_{n-3}\>\cdots\>P_1\>P_0\>M\>0
\]
in $\Mcc$. By  Lemma \ref{lem3.5}, we have the following $\xi$-exact complex
\[
	0\>N\>X\>P_{n-2}\>L\>0
\]
with $X\in\Mcx$ and $P_{n-2}\in\Px$. Thus, we can get a $\xi$-exact complex
\[
	0\>N\>X\>P_{n-2}\>\cdots\>P_1\>P_0\>M\>0
\]
with $X\in\Mcx$ and $P_i\in\Px$ for  $0\leq i\leq n-2$, as desired.

(2) By (1) and (\ref{HT}), there exists a $\xi$-exact complex
\[
	0\>N\>X'\>P_{n-2}\>\cdots\>P_1\>P_0\>M\>0
\]
with $X'\in\Mcx$ and $P_i\in\Px$ for  $0\leq i\leq n-2$. Then we have a $\xi$-exact complex
\[
	0\>K\>P_{n-2}\>\cdots\>P_1\>P_0\>M\>0
\]
and an $\Mbe$-triangle $N\>X'\>K\dashrightarrow$ in $\xi$. Note that $X'\in\Mcx$, there exists an $\Mbe$-triangle $X\>P_{n-1}\>X'\dashrightarrow$ in $\xi$ with $X\in\Mcx$ and $P_{n-1}\in\Px$. By Lemma \ref{Thm3.2} and $\rm(ET4)^{op}$, we have the following commutative diagram 
\[
		\xymatrix{
			X\ar@{=}[d] \ar[r] & L \ar[r]\ar[d] & N \ar[d] \ar@{-->}[r] &  \\
			X\ar[r] & P_{n-1}\ar[d] \ar[r] & X'\ar[d] \ar@{-->}[r]& \\
			& K\ar@{=}[r]\ar@{-->}[d]&K\ar@{-->}[d] \\
			& & &  }
\] 
where all  rows and columns are  $\mathbb{E}$-triangles in $\xi$. Thus, we obtain a $\xi$-exact complex
\[
	0\>L\>P_{n-1}\>P_{n-2}\>\cdots\>P_1\>P_0\>M\>0
\]
and an $\Mbe$-triangle $X\>L\>N\dashrightarrow$ in $\xi$ with $X\in\Mcx$ and $P_i\in\Px$ for all $0\leq i\leq n-1$.

(3) We will prove the result by induction on $n$. When $n=0$, since $M\in\Mcx$, there is a $\Mbe$-triangle $X'\>P\>M$ in $\xi$ with $X'\in\Mcx$ and $P\in\Px$. By  Lemma \ref{BH} (1) and Lemma \ref{Thm3.2}, we have the following commutative diagram 
\[
	\xymatrix{
		&N\ar@{=}[r]\ar[d]&N\ar[d]\\
		X'\ar[r]\ar@{=}[d]&X\ar[r]\ar[d]&X_0\ar[d]\ar@{-->}[r]&\\
		X'\ar[r]&P\ar[r]\ar@{-->}[d]&M\ar@{-->}[r]\ar@{-->}[d]&\\
		&&&
	}
\]	
made of $\mathbb{E}$-triangles in $\xi$. Since $\Mcx$ is closed under $\xi$-extensions, $X$  lies in $\Mcx$. And the $\Mbe$-triangle $N\>X\>P\dashrightarrow$ is what we need. Now we assume $n\geq 1$. By (\ref{HT}), there exist an $\Mbe$-triangle $N\>X_{n-1}\>K\dashrightarrow$ in $\xi$ and a $\xi$-exact complex
\[
	0\>K\>X_{n-2}\>X_{n-3}\>\cdots\>X_1\>X_0\>M\>0
\]
in $\Mcc$. By induction hypothesis, there is a $\xi$-exact complex 
\[
	0\>K\>X'\>P_{n-3}\>\cdots\>P_2\>P_1\>P_0\>0
\]
with $X'\in\Mcx$ and $P_i\in\Px$ for any $0\leq i\leq n-3$. So we can obtain a new $\xi$-exact complex 
\[
	0\>N\>X_{n-1}\>X'\>P_{n-3}\>\cdots\>P_2\>P_1\>P_0\>0
\]
in $\Mcc$. Then there are two $\xi$-exact complex 
\[
	0\>L\>P_{n-3}\>\cdots\>P_1\>P_0 \text{\ and\ } N\>X_{n-1}\>X'\>L\>0
\]
in $\Mcc$. Thus, there is  a $\xi$-exact complex 
\[
	0\>N\>X\>P_{n-2}\>L\>0
\]
with $X\in\Mcx$ and $P_{n-2}\in\Px$ by (1). So there exists  a $\xi$-exact complex 
\[
	0\>N\>X\>P_{n-2}\>\cdots\>P_2\>P_1\>P_0\>0
\]
with $X\in\Mcx$ and $P_i\in\Px$ for all $0\leq i\leq n-2$ in $\Mcc$.

(4) The proof is similar to that of (2).

(5) From (\ref{HT}), there exist two $\xi$-exact complex
\[
	0\>N\>X_{n-1}\>X_{n-2}\>\cdots\>X_{m+1}\>K_m\>0
\]
and 
\[
	0\>K_m\>X_m\>X_{m-1}\>\cdots\>X_0\>M\>0
\]
in $\Mcc$. By (1) and its dual, there are two $\xi$-exact complex
\[
	0\>N\>H_{n-1}\>H_{n-2}\>\cdots\>H_{m+2}\>X'\>K_m\>0
\]
and 
\[
	0\>K_m\>X''\>P_{m-1}\>\cdots\>P_1\>P_0\>M\>0
\]
with $X',X''\in\Mcx$, $H_i\in \Mch$ for $m+2\leq i\leq n-1$ and $P_j\in\Px$ for $0\leq j\leq m-1$. Then we can get the following $\xi$-exact complex 
\[
	0\>N\>H_{n-1}\>\cdots\>H_{m+2}\>X'\>X''\>P_{m-1}\>\cdots\>P_0\>M\>0
\]
in $\Mcc$. So there exist three $\xi$-exact complexes in $\Mcc$ as follows
\begin{gather*}
	0\>N\>H_{n-1}\>\cdots\>H_{m+2}\>L'\>0\\
	0\>L'\>X'\>X''\>L''\>0\\
	0\>L''\>P_{m-1}\>\cdots\>P_0\>M\>0. 
\end{gather*}
It follows from the dual of (1), one can obtain a $\xi$-exact complex
\[
	0\>L'\>H_{m+1}\>X\>L''\>0
\]
with $H_{m+1}\in\Mch$ and $X\in\Mcx$. Consider the related $\xi$-exact complexes, we have the following $\xi$-exact complex
\[
	0\>N\>H_{n-1}\>\cdots\>H_{m+1}\>X\>P_{m-1}\>\cdots\>P_0\>M\>0,
\]
as desired.

(6) The proof is similar to that of (5).
\end{proof}

\section{Quasi-resolving resolution dimensions}
\quad~ In this section, we investigate the properties and consequences of $\Mcx$-resolution dimension for a quasi-resolving subcategory $\Mcx$ in an extriangulated category $\Mcc$.  Moreover, we provide several equivalent characterizations for objects with finite $\Mcx$-resolution dimension.

\begin{lem}\label{THC}
Let $\Mcx$ be a quasi-resolving subcategory of $\Mcc$. For any object $M\in \res{\Mcx}$, if there are two $\xi$-exact complexes
\[
	0\>X_n\>X_{n-1}\>\cdots \> X_1\>X_0\>M\>0
\]
and 
\[
	0\>Y_n\>Y_{n-1}\>\cdots \> Y_1\>Y_0\>M\>0
\]
with $X_i$ and $Y_i$ in $\Mcx$ for all $0\leq i\leq n-1$, then $X_n\in\Mcx$ if and only if $Y_n\in\Mcx$.
\end{lem}
\begin{proof}
Since $M\in \res{\Mcx}$, $M$ is in $M\in \res{\Px}$ by Proposition \ref{ZGQ}. Then we have the following $\xi$-exact complex
\[
	0\>K_n\>P_{n-1}\>\cdots\> P_1\>P_0\>M\>0
\]
with $P_i\in\Px$ for all $0\leq i\leq n-1$. Consider the $\xi$-exact complex 
\[
	0\>X_n\>X_{n-1}\>\cdots \> X_1\>K_1\>0
\] 
in $\Mcc$ and the $\Mbe$-triangle 
\[
K_1\>X_0\>M\dashrightarrow
\] 
in $\xi$. By repeatedly applying \cite[Lemma 4]{JDPPR}, we can obtain the following $\xi$-exact complex 
\[
	0\>K_n\>X_n\oplus P_{n-1}\>X_{n-1}\oplus P_{n-2}\>\cdots \> X_1\oplus P_0\>X_0\>0.
\]
Similarly, we have the following $\xi$-exact complex 
\[
	0\>K_n\>Y_n\oplus P_{n-1}\>Y_{n-1}\oplus P_{n-2}\>\cdots \> Y_1\oplus P_0\>Y_0\>0.
\]
Since $\Mcx$ is quasi-resolving, then there are two $\Mbe$-triangles
\[
K_n\>X_n\oplus P_{n-1}\>X\dashrightarrow \text{\ and \ } K_n\>Y_n\oplus P_{n-1}\>Y\dashrightarrow
\]
in $\xi$ with $X,Y\in\Mcx$. Thus, we have $X_n\oplus P_{n-1}\in\Mcx$ if and only if $K_n\in\Mcx$ if and only if $Y_n\oplus P_{n-1}\in\Mcx$. So $X_n\in\Mcx$ if and only if $Y_n\in\Mcx$.
\end{proof}

\begin{prop}\label{Prop4.2}
Let $\Mcx$ be a quasi-resolving subcategory of $\Mcc$ and $M\in\Mcx$, then the following are equivalent for any  $n\geq 0$.

(1) $\resdim{\Mcx}{M}\leq n$.

(2) $\Omega^{n+i}_{\Px}(M)\in \Mcx$, $\forall i\geq 0$.

(3) $\Omega^{n+i}_{\Mcx}(M)\in \Mcx$, $\forall i\geq 0$.
\end{prop}
\begin{proof}
	It  follows directly from Lemma \ref{THC}.
\end{proof}

Now we can compare $\Mcx$-resolution dimensions of an $\Mbe$-triangle in $\xi$ as follows. 

\begin{prop}\label{Prop4.3}
Let $\Mcx$ be a quasi-resolving subcategory of $\Mcc$, and let 
\[
A\>B\>C\dashrightarrow
\]
be an $\Mbe$-triangle in $\xi$. Then we have the following statemants.

(1) $\resdim{\Mcx}{B}\leq \max\left\{\resdim{\Mcx}{A},\resdim{\Mcx}{C}\right\}$.

(2) $\resdim{\Mcx}{A}\leq \max\left\{\resdim{\Mcx}{B},\resdim{\Mcx}{C}-1\right\}$.

(3) $\resdim{\Mcx}{C}\leq \max\left\{\resdim{\Mcx}{A}+1,\resdim{\Mcx}{B}\right\}$.
\end{prop}
\begin{proof}

We only need to prove the case where the right-hand side of the inequality is finite, as the result is trivial when it is infinite. For any $A\in\Mcc$ with $\resdim{\Mcx}{A}=m$, there is the following $\xi$-exact complex 
\[
0\>P^A_m\>P^A_{m-1}\>\cdots\>P^A_1\>P^A_0\>A\>0
\] 
in $\Mcc$, where $P^A_i\in\Px$ for $0\leq i\leq m-1$ and $P^A_m\in\Mcx$ by Proposition \ref{Prop4.3}.

(1) Let $r=\max\left\{\resdim{\Mcx}{A},\resdim{\Mcx}{C}\right\}$, then by repeatedly applying \cite[Lemma 4]{JDPPR}, we can obtain the following $\xi$-exact complex 
\[
	0\>P^A_r\oplus P^C_r\>P^A_{r-1}\oplus P^C_{r-1}\>\cdots \> P^A_0\oplus P^C_0\>B\>0
\]
in $\Mcc$ with $P^A_r\oplus P^C_r\in\Mcx$. Thus, $\resdim{\Mcx}{B}\leq r$ by definition and the desired assertion are obtained.

(2) Let $r=\max\left\{\resdim{\Mcx}{B},\resdim{\Mcx}{C}-1\right\}$, then by \cite[Theorem 1]{JDPPR}, we can obtain the following $\xi$-exact complex 
\[
	0\>P^C_r\oplus P^B_{r-1}\>P^C_{r-1}\oplus P^B_{r-2}\>\cdots P^C_2\oplus P^B_1\>K\>A\>0
\]
in $\Mcc$ with $P^A_r\oplus P^C_r\in\Mcx$ and $K\in\Mcx$. Thus, $\resdim{\Mcx}{A}\leq r$ by definition and the desired assertion are obtained.

(3) It is similar to the proof of (2).
\end{proof}

\begin{cor}\label{Cor4.4}
	Let $\Mcx$ be a quasi-resolving subcategory of $\Mcc$, and let 
	\[
	A\>B\>C\dashrightarrow
	\]
be an $\Mbe$-triangle in $\xi$. Then we have the following statemants.
\end{cor}
(1) If $A$ is in $\Mcx$ and neither $B$ nor $C$ in $\Mcc$, then $\resdim{\Mcx}{B}=\resdim{\Mcx}{C}$.

(2) If $B$ is in $\Mcx$, then either $A\in\Mcx$ or else $\resdim{\Mcx}{A}=\resdim{\Mcx}{C}-1$.

(3) If $C$ is in $\Mcx$, then $\resdim{\Mcx}{A}=\resdim{\Mcx}{B}$.

(4) If two of $A, B$ and $C$ are in $\widehat{X}$, then so is the third.


\begin{prop}\label{Prop4.5}
Let $\Mcx$ be a quasi-resolving subcategory of $\Mcc$ and $\Mch$ be a $\xi$-cogenerator of $\Mcx$. Then for any $M\in\Mcc$ with $\resdim{\Mcx}{M}=n<\infty$, there exist two $\Mbe$-triangles
\[
K\>X\>M\dashrightarrow \text{\ and \ } M\>W\>X'\dashrightarrow
\]
in $\xi$ with $X,X'\in\Mcx$, $\resdim{\Mch}{K}=n-1$ and $\resdim{\Mch}{W}=\resdim{\Mcx}{W}=n$. 
\end{prop}
\begin{proof}
 By the dual of Theorem \ref{thm3.6} (3) and (4), there exist two $\Mbe$-triangles \[
K\>X\>M\dashrightarrow \text{\ and \ } M\>W\>X'\dashrightarrow
\]
in $\xi$ with $X,X'\in\Mcx$ and $\resdim{\Mch}{K}\leq n-1$, $\resdim{\Mch}{W}\leq n$. Note that $\resdim{\Mcx}{M}=n$ and Corollary \ref{Cor4.4} (2) and (3), we can get $\resdim{\Mch}{K}=n-1$ and $\resdim{\Mch}{W}=n$.
\end{proof}

\begin{thm}\label{Thm4.6}
Let $\Mcx$ be a quasi-resolving subcategory of $\Mcc$ and $\Mch$ be a $\xi$-cogenerator of $\Mcx$. Then for any object $M$ with $\resdim{\Mcx}{M}<\infty$, the following conditions are equivalent for any  $n\geq 0$:

	(1) $\resdim{X}{M}\leq n$.

	(2) There is a $\xi$-exact complex 
\[
0\>H_n\>H_{n-1}\>\cdots\> H_1\> X\>M\>0
\]
with $X\in\Mcx$ and $H_i\in\Mch$ for $1\leq i\leq n$.

   (3) There is an $\Mbe$-triangle 
\[
M\>W\>X\dashrightarrow 
\]
in $\xi$ with $X\in\Mcx$ and $\resdim{\Mch}{W}\leq n$.

  (4) For any $0\leq m\leq n$, there exists a $\xi$-exact complex 
\[
0\>H_n\>\cdots\> H_{m+1}\>X\>P_{m-1}\>\cdots\> P_1\> P_0\>M\>0
\]
with $X\in\Mcx$, $H_i\in\Mch$ for $m+1\leq i\leq n$ and $P_j\in\Px$ for $0\leq j\leq m-1$.
\end{thm}
\begin{proof}
	(1) $\Leftrightarrow$ (2) by the dual of Theorem \ref{thm3.6} (1). 
	
	(1) $\Leftrightarrow$ (3) by Corollary \ref{Cor4.4} (3) and Proposition \ref{Prop4.5}.

	(1) $\Leftrightarrow$ (4)  by Theorem \ref{thm3.6} (6). 
\end{proof}

\begin{cor}
Let $\Mcx$ be a quasi-resolving subcategory of $\Mcc$ and $\Mch$ be a $\xi$-cogenerator of $\Mcx$ such that $\Mch\subseteq \Px$. Then for any object $M$ of $\Mcc$ with  $\resdim{\Mcx}{M}<\infty$, $\resdim{X}{M}\leq n$ if and only if for any $0\leq m\leq n$ there is a $\xi$-exact complex 
\[
0\>P_n\>\cdots\> P_{m+1}\>X\>P_{m-1}\>\cdots\> P_1\> P_0\>M\>0
\]
with $X\in\Mcx$ and $P_i\in\Px$ for any  $i\neq m$.
\end{cor}

\begin{lem}\label{Lem4.8}
Let $\Mcx$ and $\mathcal{Y}$ be two subcategories of $\Mcc$.

(1) If $\Mcx\perp \mathcal{Y}$, then $\Mcx\perp \widehat{\mathcal{Y}}$. In particular, if $\mathcal{Y}\perp \mathcal{Y}$, then $\mathcal{Y}\perp \widehat{\mathcal{Y}}$.

(2) If $M\in{}^{\perp}\mathcal{Y}$, then $M\in{}^{\perp}\widehat{\mathcal{Y}}$.
\end{lem}
\begin{proof}
(1) Let $M\in\widehat{\mathcal{Y}}$ and $\resdim{\mathcal{Y}}{M}=n$, then there is a $\xi$-exact complex
\[
0\>Y_n\>Y_{n-1}\>\cdots\>Y_1\>Y_0\>M\>0
\]
with $Y_i\in\mathcal{Y}$ and the $\xi$-resolution  $\mathbb{E}$-triangle $K_{i+1}\>Y_i\>K_i\dashrightarrow$ (set $K_0=M$ and $K_n=Y_n$) for any  $0\leq i\leq n$. By Lemma \ref{LZHL}, we have the following exact sequence 
\[
\cdots\to \ext^m(X,Y_i)\to \ext^m(X,K_i)\to \ext^{m+1}(X,K_{i+1})\to \ext^{m+1}(X,Y_i)\to\cdots
\]
for any $m\geq 1$ and $X\in\Mcx$. Note that  $\Mcx\perp \mathcal{Y}$, so $\ext^m(X,Y_i) =\ext^{m+1}(X,Y_i)=0$, and it implies that $\ext^m(X,K_i)\cong \ext^{m+1}(X,K_{i+1})$. Thus, we have 
\[
	\ext^i(X,M)\cong\ext^{n+i}(X,Y_n)=0
\]
for any $i\geq 1$ and $X\in\Mcx$. Therefore, we get that  $\Mcx\perp \widehat{\mathcal{Y}}$.

(2) It follows directly from (1).
\end{proof}

\begin{lem}\label{Lem4.9}
Let $\Mch$ be a $\Ext$-injective $\xi$-cogenerator of $\Mcx$ such that  it is either closed under $\xi$-extensions and $\xi$-cocones, or closed under direct summands. Then $\Mch=\widehat{\Mch}\cap{}^{\perp}\Mch$.
\end{lem}
\begin{proof}
Obviously, $\Mch\subseteq\widehat{\Mch}\cap{}^{\perp}\Mch$. Conversely, let $M\in \widehat{\Mch}\cap{}^{\perp}\Mch$ and $\resdim{\Mch}{M}=n$, then there exists a $\xi$-exact complex
\[
0\>H_n\>H_{n-1}\>\cdots\>H_1\>H_0\>M\>0
\]
with all $H_i\in\Mch$ and the $\xi$- resolution $\mathbb{E}$-triangle $K_{i+1}\>H_i\>K_i\dashrightarrow$ (set $K_0=M$ and $K_n=H_n$) for any  $0\leq i\leq n$. By Lemma \ref{LZHL}, there is the following exact sequence
\[
	\cdots\to\ext^{k}(M,H)\to\ext^{k}(H_0,H)\to\ext^{k}(K_1,H)\to\ext^{k+1}(M,H)\to\cdots
\]
for any $H\in\Mch$. Note that $\ext^{k}(M,H)=\ext^{k}(H_0,H)=\ext^{k+1}(M,H)=0$, so $\ext^{k}(K_1,H)=0$ for any $k\geq 1$ i.e. $K_1\in {}^{\perp}\Mch$. Repeating this process, we can obtain $K_i\in {}^{\perp}\Mch$ for any $i\geq 1$. Consider the following $\Mbe$-triangle 
\[
H_n\>H_{n-1}\>K_{n-1}\dashrightarrow
\]
in $\xi$. Since $\Mch$ is a  $\Ext$-injective $\xi$-cogenerator of $\Mcx$, there is the following exact sequence
\[
\Mcc(H_{n-1},H_n)\to \Mcc(H_n,H_n)\to \ext^1(K_{n-1},H_n)=0.
\]
So the $\Mbe$-triangle $H_n\>H_{n-1}\>K_{n-1}\dashrightarrow$ is  split. It follows that $H_{n-1}\cong H_n\oplus K_{n-1}$ and there exists a split $\Mbe$-triangle  
\[
K_{n-1}\>H_{n-1}\>H_n\dashrightarrow
\] 
in $\xi$. Since $\Mch$ is either closed under $\xi$-extensions and $\xi$-cocones, or closed under direct summands by assumption, we have $K_{n-1}\in\Mch$. Repeating this process, we can obtain $K_i\in\Mch$ for any $i\leq n$, hence $M\in\Mch$ and $\widehat{\Mch}\cap{}^{\perp}\Mch\subseteq \Mch$. Thus, $\Mch=\widehat{\Mch}\cap{}^{\perp}\Mch$.
\end{proof}

Our main result in this section is the following.

\begin{thm}
Let $\Mcx$ be a quasi-resolving subcategory of $\Mcc$ and $\Mch$ be a $\Ext$-injective $\xi$-cogenerator of $\Mcx$ such that $\Mch$ is closed under $\xi$-extensions and $\xi$-cocones, or $\Mch$ is only closed under direct summands. Then for any object $M$ of $\Mcc$  with $\resdim{\Mcx}{M}<\infty$, the following conditions are equivalent for any  $n\geq 0$:

(1) $\resdim{\Mcx}{M}\leq n$.

(2) $\Omega_{\Px}^{n+i}(M)\in\Mcx$, $\forall i\geq 0$.

(3) $\Omega_{\Mcx}^{n+i}(M)\in\Mcx$, $\forall i\geq 0$.

(4) $\ext^{n+i}(M,H)=0$, $\forall i\geq 1$, $\forall H\in\Mch$.

(5) $\ext^{n+i}(M,N)=0$, $\forall i\geq 1$, $\forall N\in\widehat{\Mch}$.

(6) There exists an $\Mbe$-triangle 
\[
M\>W\>X\dashrightarrow 
\]
in $\xi$ with $X\in\Mcx$ and $\resdim{\Mch}{W}\leq n$.

(7) There exists an $\Mbe$-triangle 
\[
K\>X'\>M\dashrightarrow 
\]
in $\xi$ with $X'\in\Mcx$ and $\resdim{\Mch}{K}\leq n-1$.

(8) There exist two $\Mbe$-triangles 
\[
W_M\>X_M\>M\dashrightarrow \text{\  and \ } M\>W^M\>X^M\dashrightarrow 
\]
in $\xi$ such that  $X_M, X^M$ belong to $\Mcx$ and $\resdim{\Mch}{W_M}\leq n-1$, $\resdim{\Mch}{W^M}=\resdim{\Mcx}{W^M}\leq n$. 

(9) There is a $\xi$-exact complex 
\[
0\>H_n\>H_{n-1}\>\cdots\> H_1\> X\>M\>0
\]
with $X\in\Mcx$ and $H_i\in\Mch$ for $1\leq i\leq n$.

(10) For any $0\leq m\leq n$, there exists a $\xi$-exact complex 
\[
0\>H_n\>\cdots\> H_{m+1}\>X\>P_{m-1}\>\cdots\> P_1\> P_0\>M\>0
\]
with $X\in\Mcx$, $H_i\in\Mch$ for $m+1\leq i\leq n$ and $P_j\in\Px$ for $0\leq j\leq m-1$.
\end{thm}
\begin{proof}
(1) $\Leftrightarrow$ (2) $\Leftrightarrow$ (3) by Proposition \ref{Prop4.2}.

(1) $\Rightarrow$ (4) Suppose that $\resdim{\Mcx}{M}\leq n$, then there exists a $\xi$-exact complex
\[
0\>X_n\>X_{n-1}\>\cdots\>X_1\>X_0\>M\>0
\]
with all $X_i\in\Mcx$ and the $\xi$-resolution $\mathbb{E}$-triangle $K_{i+1}\>X_i\>K_i\dashrightarrow$ (set $K_0=M$ and $K_n=X_n$) for any  $0\leq i\leq n$. By Lemma \ref{LZHL}, there is the following exact sequence
\[
\cdots\to\ext^{k}(X_i,H)\to\ext^{k}(K_{i+1},H)\to\ext^{k+1}(K_i,H)\to\ext^{k+1}(X_i,H)\to\cdots
\]
for any $H\in\Mch$. Since $\Mch$ is a $\Ext$-injective $\xi$-cogenerator of $\Mcx$, we have $\ext^{k\geq 1}(X_i,H)=0$. Thus, there is $\ext^{k+1}(K_i,H)\cong\ext^{k}(K_{i+1},H)$ for any $k\geq 1$. So $\ext^{n+i}(M,H)\cong \ext^{i}(X_n,H)=0$ for any $i\geq 1$.

(4) $\Rightarrow$ (1) Since $M$ has finite $\Mcx$-resolution dimension, there exists an $\Mbe$-triangle
\[
K\>X\>M\dashrightarrow 
\]
in $\xi$ with $X\in\Mcx$ and $\resdim{\Mch}{K}<\infty$ by Proposition \ref{Prop4.5}. By Lemma \ref{LZHL}, there is the following exact sequence
\[
\cdots\to\ext^{k}(X,H)\to\ext^{k}(K,H)\to\ext^{k+1}(M,H)\to\ext^{k+1}(X,H)\to\cdots
\]
for any $H\in\Mch$ and $k\geq 1$. So $\ext^{k}(K,H)\cong\ext^{k+1}(M,H)$ by $\ext^{k\geq 1}(X,H)=0$. Hence, we obtain $\ext^{n+i}(K,H)=0$ for any $i\geq 0$. Note that $\resdim{\Mch}{K}<\infty$, and let $\resdim{\Mch}{K}=m$. We only need to consider the case where $m\geq n$. Then there exists a $\xi$-exact complex
\[
0\>H_m\>H_{m-1}\>\cdots\>H_1\>H_0\>K\>0
\]
with all $H_i\in\Mch$ and the $\xi$- resolution $\mathbb{E}$-triangle $K_{i+1}\>H_i\>K_i\dashrightarrow$ (set $K_0=K$ and $K_m=H_m$) for any  $0\leq i\leq m$. By Lemma \ref{LZHL}, for all $k\geq 1$, there is the following exact sequence
\[
\cdots\to \ext^k(H_i,H)\to\ext^k(K_{i+1},H)\to\ext^{k+1}(K_i,H)\to \ext^{k+1}(H_i,H) \to\cdots
\]
Since $\ext^k(H_i,H)=\ext^{k+1}(H_i,H)=0$, we have $\ext^k(K_{i+1},H)\cong\ext^{k+1}(K_i,H)$. Thus, we obtain  $\ext^i(K_{n-1},H)\cong\ext^{i+n-1}(K,H)=0$ for any $i\geq 1$, which means $K_{n-1}\in{}^{\perp} \Mch$. Note that $K_{n-1}\in\widehat{\Mch}$, so $K_{n-1}$ is in $\Mch$ by Lemma \ref{Lem4.9}. Therefore, $\resdim{\Mch}{K}\leq n-1$ and $\resdim{\Mcx}{M}\leq n$.

(4) $\Rightarrow$ (5) The proof is similar to that of Lemma \ref{Lem4.8}.

(5) $\Rightarrow$ (4) It is obvious.

(6) $\Rightarrow$ (7) Since $\resdim{\Mch}{W}\leq n$, there is an $\Mbe$-triangle $K\>H\>W\dashrightarrow$ in $\xi$ with $H\in\Mch$ and $\resdim{\Mch}{K}\leq n-1$. It follows from Lemma \ref{Thm3.2} and $\rm(ET4)^{op}$ that we have the following commutative diagram 
\[
		\xymatrix{
			K\ar@{=}[d] \ar[r] & X' \ar[r]\ar[d] & M \ar[d] \ar@{-->}[r] &  \\
			K\ar[r] & H\ar[d] \ar[r] & W\ar[d] \ar@{-->}[r]& \\
			& X\ar@{=}[r]\ar@{-->}[d]&X\ar@{-->}[d] \\
			& & &  }
\] 
where all rows and columns are  $\mathbb{E}$-triangles in $\xi$. Since $\Mcx$ is closed under $\xi$-cocones, we conclude that $X'$ is in $\Mcx$. And the top row gives the required $\Mbe$-triangle.

(7) $\Rightarrow$ (6) There exists an $\Mbe$-triangle $X'\>H\>X\dashrightarrow$ in $\xi$ with $X\in\Mcx$ and $H\in\Mch$, since $\Mch$ is a $\Ext$-injective $\xi$-cogenerator of $\Mcx$. It follows from Lemma \ref{Thm3.2} and $\rm(ET4)$ that we have the following commutative diagram 
\[
		\xymatrix{
			K\ar@{=}[d] \ar[r] & X' \ar[r]\ar[d] & M \ar[d] \ar@{-->}[r] &  \\
			K\ar[r] & H\ar[d] \ar[r] & W\ar[d] \ar@{-->}[r]& \\
			& X\ar@{=}[r]\ar@{-->}[d]&X\ar@{-->}[d] \\
			& & &  }
\] 
where all rows and columns are  $\mathbb{E}$-triangles in $\xi$. Since $\resdim{\Mch}{K}\leq n-1$, there is $\resdim{\Mch}{W}\leq n$ by definition. Thus, the third column gives the required $\Mbe$-triangle.

(1) $\Leftrightarrow$ (8) It follows from Proposition \ref{Prop4.5} and Theorem \ref{Thm4.6}.

(1) $\Leftrightarrow$ (6) $\Leftrightarrow$ (9) $\Leftrightarrow$ (10) It follows from Theorem \ref{Thm4.6}.
\end{proof}

\section{Gorenstein quasi-resolving subcategories}
\quad~ In this section, we will construct a new quasi-resolving subcategory from a given quasi-resolving subcategory, which generalizes the notion of $\xi$-Gorenstein projective objects given by Hu, Zhang and Zhou in \cite{JDP}. By applying the previous results to this subcategory, we obtain some known results.   

We assume that $\Mcx$ is always a  quasi-resolving subcategory throughout this section.

\begin{defn}
	A complete $\mathcal{P_{\Mcx}}(\xi)$-resolution is a $\Mcc(-,\Mcx)$-exact complex in $\Mcc$
\[
\mathbf{P}:\cdots \> P_2\>P_1\>P_0\>P_{-1}\>P_{-2}\>\cdots
\]
where $P_n$ is in $\Px$ for each integer $n$.  In this case, there exists a $\Mcc(-,\Mcx)$-exact $\mathbb{E}$-triangle  $K_{n+1}\stackrel{g_n}\longrightarrow X_n\stackrel{f_n}\longrightarrow K_n\stackrel{\delta_n}\dashrightarrow$ in $\xi$ which is the $\xi$-resolution $\mathbb{E}$-triangle of $\mathbf{P}$ for any integer $n$. Then the objects $K_n$ are called $\mathcal{GQ_{\Mcx}(\xi)}$-projective for each integer $n$. The subcategory $\Gpx$ of all $\mathcal{GQ_{\Mcx}(\xi)}$-projective objects in $\Mcc$ is called Gorenstein quasi-resolving.
\end{defn}

\begin{rem}
	(1)  Obviously, $\Px$ is a  $\xi$-cogenerator of $\Gpx$ with  is closed under $\xi$-extensions and $\xi$-cocones.
	
	(2) If $\Mcx=\P$, then $\Gpx$ is just the subcategory $\GP$ consisting of all $\xi$-Gorenstein projective objects in $\Mcc$ (see \cite[Definition 4.8]{JDP}).
\end{rem}

\begin{prop}\label{Prop5.3}
	An object $M$ is in $\Px$ if and only if $M\in\Mcx\cap \Gpx$.
\end{prop}
\begin{proof}
The ``only if'' part is obvious. For the ``if'' part, let $M\in\Gpx$, then there is an $\Mbe$-triangle 
\[
	M\> P \> K\dashrightarrow 
\]
in $\xi$ with $P\in\Px$ and $K\in\Gpx$, which is $\Mcc(-,\Mcx)$-exact  by definition. Note that $M\in\Mcx$, then we have the following exact sequence
\[
0\>\Mcc(K,M)\>\Mcc(P,M)\>\Mcc(M,M)\>0.
\]
Hence the $\Mbe$-triangle $M\> P \> K\dashrightarrow$ is split. So we obtain that  $M\oplus K\cong P\in\P$, and it implies that $M\in\P$. Then $M\in\P\cap\Mcx=\Px$.
\end{proof}

\begin{cor}
 $\Mcx=\Px$ if and only if $\Mcx\subseteq \Gpx$. 
\end{cor}

\begin{lem}\label{HHH}
	Let $A\stackrel{x}\longrightarrow B \stackrel{y}\longrightarrow C  \stackrel{\delta}\dashrightarrow$  be a $\mathbb{E}$-triangle in $\xi$.
	
(1) If $C\in \Gpx$, then the $\mathbb{E}$-triangle $A\stackrel{x}\longrightarrow B \stackrel{y}\longrightarrow C  \stackrel{\delta}\dashrightarrow$ is $\Mcc(-,\Mcx)$-exact.
	
(2) If $C$ is a direct summand of $\mathcal{GQ_{\Mcx}(\xi)}$-projective object, then the $\mathbb{E}$-triangle $A\stackrel{x}\longrightarrow B \stackrel{y}\longrightarrow C  \stackrel{\delta}\dashrightarrow$ is $\Mcc(-,\Mcx)$-exact.
	\end{lem}
	
	\begin{proof}
(1) Since  $C$ is in $\Gpx$, there is a $\Mcc(-,\Mcx)$-exact $\mathbb{E}$-triangle  $K \stackrel{f}\longrightarrow P \stackrel{g}\longrightarrow C  \stackrel{\theta}\dashrightarrow$ in $\xi$ with $P\in\Px$ and $K\in \Gpx$. Hence there exists a commutative diagram 
	\[
	\xymatrix{
		&K\ar@{=}[r]\ar[d]^{f_1}&K\ar[d]^f\\
		A\ar[r]^{x_1}\ar@{=}[d]&M\ar[r]^{y_1}\ar[d]^{g_1}&P\ar[d]^g\ar@{-->}[r]^{g^*\delta}&\\
		A\ar[r]^x&B\ar[r]^y\ar@{-->}[d]^{y^*\theta}&C\ar@{-->}[r]^{\delta}\ar@{-->}[d]^{\theta}&\\
		&&&
	}
	\]	
made of $\mathbb{E}$-triangles by Lemma \ref{BH} (1). Note that $P$ is $\xi$-projective, so $g$ factors through $y$. Then  the $\mathbb{E}$-triangle $A\stackrel{x_1}\longrightarrow M\stackrel{y_1}\longrightarrow P \stackrel{g^*\delta}\dashrightarrow $ is split by \cite[Corollary 3.5]{HY}, hence it is a $\Mcc(-,\Mcx)$-exact $\mathbb{E}$-triangle in $\xi$. Therefore, the $\mathbb{E}$-triangle 
	$A\stackrel{x}\longrightarrow B \stackrel{y}\longrightarrow C  \stackrel{\delta}\dashrightarrow$
	 is $\Mcc(-,\Mcx)$-exact by \citep[Lemma 4.10 (1)]{JDP}.
	
(2) Suppose that $C\oplus C'$ is a $\mathcal{GQ_{\Mcx}(\xi)}$-projective object, then we have the following commutative diagram 
	 \[
	 \xymatrix{
		 &C'\ar@{=}[r]\ar[d]^{f_1}&C'\ar[d]^{\begin{tiny}\begin{bmatrix}
			0\\1
			\end{bmatrix}\end{tiny}}\\
		 A\ar[r]^{x_1}\ar@{=}[d]&M\ar[r]^{y_1}\ar[d]^{g_1}&C\oplus C'\ar[d]^{\begin{tiny}\begin{bmatrix}
			1&0
			\end{bmatrix}\end{tiny}}\ar@{-->}[r]^{\qquad\begin{tiny}\begin{bmatrix}
				1&0
				\end{bmatrix}\end{tiny}^*\delta}&\\
		 A\ar[r]^x&B\ar[r]^y\ar@{-->}[d]^{y^*\theta}&C\ar@{-->}[r]^{\delta}\ar@{-->}[d]^{\theta}&\\
		 &&&
	 }
\]	
made of $\mathbb{E}$-triangles in $\xi$ by Lemma \ref{BH} (1) and Lemma \ref{Thm3.2}. Note that the second horizontal is $\Mcc(-,\Mcx)$-exact by (1) and the third vertical  is  $\Mcc(-,\Mcx)$-exact, then so is the third horizontal by \citep[Lemma 4.10 (1)]{JDP}.
\end{proof}

In the following part, we give some characterizations of $\mathcal{GQ_{\Mcx}(\xi)}$-projective objects.

\begin{lem}\label{TTTT}
	Assume that $G$ is an object  in $\Gpx$, then $\ext^0(G,X) \simeq \Mcc(G,X)$ and $\ext^{i\geq 1}(G,X)=0$ for any  $X\in\Mcx$. 
\end{lem}
\begin{proof}
Since $G\in\Gpx$, there exists an $\Mbe$-triangle $G'\> P \> G  \dashrightarrow$ in $\xi$ with $G'\in\Gpx$ and $P\in\Px$. If $X\in\Mcx$, then we have the following commutative diagram 
\[
	\xymatrix{
		0\ar[r]&\Mcc(G,X)\ar[r]\ar[d]_{\varphi_1}&\Mcc(P,X)\ar[r]\ar[d]^{\varphi_2}_{\simeq }&\Mcc(G',X)\ar[r]\ar[d]^{\varphi_3}& 0&\\
	0\ar[r]&\ext^0(G,X)\ar[r]&\ext^0(P,X)\ar[r]&\ext^0(G',X)\ar[r]&\ext^1(G,X)\ar[r]&0.
	}
\]
where the top exact sequence follows from Lemma \ref{HHH}(1). Note that  $\varphi_1$ and  $\varphi_3$ are monic, hence $\varphi_1$ is  epic by Snake Lemma, so $\varphi_1$ is an isomorphism. Similarly, one can get that $\varphi_3$ is an isomorphism, so $\ext^1(G,X)=0$. It is easy to show that $\ext^i(G, M)=0$ for any $i\geq 1$ by Lemma \ref{LZHL}.
\end{proof}

\begin{rem}
	$\Px$ is a $\Ext$-injective $\xi$-cogenerator of $\Gpx$. In fact, we have that $\ext^0(P,P')\cong \Mcc(P,P')$ for any $P,P'\in\Px$ since $\Px\subseteq \P$, and $\Gpx\perp\Px$ by Lemma \ref{TTTT}.
\end{rem}

\begin{lem}\label{Dec30}
	(1) If $A\stackrel{f}\longrightarrow B \stackrel{f'}\longrightarrow C  \stackrel{\delta}\dashrightarrow$ and $B\stackrel{g}\longrightarrow D\stackrel{g'}\longrightarrow E  \stackrel{\delta'}\dashrightarrow$ are both $\Mcc(-,\Mcx)$-exact   $\mathbb{E}$-triangles in $\xi$, then we have the following commutative diagram:
\[
		\xymatrix{
			A\ar@{=}[d] \ar[r]^{f} & B  \ar[r]^{f'}\ar[d]^{g} & C \ar[d]^d \ar@{-->}[r]^{\delta} &  \\
			A\ar[r]^{h} & D\ar[d]^{g'} \ar[r]^{h'} & F\ar[d]^e \ar@{-->}[r]^{\delta''} & \\
			& E\ar@{=}[r]\ar@{-->}[d]^{\delta'}&E\ar@{-->}[d]^{f'_*\delta'} \\
			& & &  }
	\] 
where all rows and columns are both $\Mcc(-,\Mcx)$-exact  $\mathbb{E}$-triangles in $\xi$.

(2) If $A\stackrel{f}\longrightarrow B \stackrel{f'}\longrightarrow C  \stackrel{\delta}\dashrightarrow$ and $D\stackrel{g}\longrightarrow C\stackrel{g'}\longrightarrow E  \stackrel{\delta'}\dashrightarrow$ are both $\Mcc(-,\Mcx)$-exact   $\mathbb{E}$-triangles in $\xi$, then we have the following commutative diagram:
\[
		\xymatrix{
			A\ar@{=}[d] \ar[r]^{d} & F  \ar[r]^{e}\ar[d]^{h} & D \ar[d]^{g} \ar@{-->}[r]^{g^*\delta} &  \\
			A\ar[r]^f & B\ar[d]^{h'} \ar[r]^{f'} & C\ar[d]^{g'} \ar@{-->}[r]^{\delta} & \\
			& E\ar@{=}[r]\ar@{-->}[d]^{\delta''}&E\ar@{-->}[d]^{\delta'} \\
			& & &  }
	\] 
where all rows and columns are both $\Mcc(-,\Mcx)$-exact  $\mathbb{E}$-triangles in $\xi$.
\end{lem}
\begin{proof}
(1) It follows from Lemma \ref{Thm3.2} and $\rm(ET4)$ that we have the desired commutative diagram where all  rows and columns are both  $\mathbb{E}$-triangles in $\xi$. We fix $X\in\Mcx$. Applying $\mathcal{C}(-,X)$ to the  diagram, one obtains the following commutative diagram of abelian groups:
\[
	\xymatrix{
		&0\ar@{..>}[d]&0\ar[d]&0\ar[d]\\
		0\ar[r]&\mathcal{C}(E,X)\ar[d]\ar@{=}[r]&\mathcal{C}(E,X)\ar[d]\ar[r] & 0\ar[d]\ar[r]&0\\
		0\ar@{..>}[r]&\mathcal{C}(F,X) \ar[d]\ar[r]^{\mathcal{C}(h',X)} & \mathcal{C}(D,X)  \ar[r]^{\mathcal{C}(h,X)}\ar[d]_{\mathcal{C}(g,X)}& \mathcal{C}(A,X) \ar@{=}[d] \ar@{..>}[r]&0&(*)  \\
		0\ar[r]&\mathcal{C}(C,X)\ar@{..>}[d]\ar[r] & \mathcal{C}(B,X)\ar[d] \ar[r]^{\mathcal{C}(f,X)} & \mathcal{C}(A,X) \ar[d] \ar[r]&0\\
		&0&0&0
			}
\]
Note that $\Mcc(h,X)$ is an epimorphism since $ \Mcc(h,X)=\Mcc(f,X)\Mcc(g,X)$. And it is easy to check that $\Mcc(h',X)$ is a monomorphism by \citep[Lemma 3(2)]{JDPPR}, so the middle row of $(*)$ is exact. Thus, the  first column of $(*)$  is also exact by $3\times 3$-Lemma.

(2) It is similar to the proof of (1).
\end{proof}

The following theorem is important in this section.

\begin{thm}\label{Thm5.8}
$\Gpx$ is a quasi-resolving subcategory of $\Mcc$.
\end{thm}
\begin{proof}
Note that $\Px\subseteq \Gpx$ by Proposition \ref{Prop5.3} and  $\Px\subseteq \P$,  we have that $\Px\subseteq \P\cap\Gpx=\mathcal{P_\Gpx}$. Since $\Gpx\subseteq\res{\Px}$, so $\Gpx\subseteq \res{\mathcal{P_\Gpx}}$. 

To prove that $\Gpx$ is quasi-resolving, it suffice to show that if
\[
A\stackrel{x}\> B \stackrel{y}\> C  \stackrel{\delta}\dashrightarrow
\]
is an $\Mbe$-triangle in $\xi$ with $C\in\Gpx$, then $A\in\Gpx$ if and only if $B\in\Gpx$.

If $A$ is in $\Gpx$, then there are two $\Mcc(-,\Mcx)$-exact $\mathbb{E}$-triangles
\[
	A\stackrel{g^A_{-1}}\longrightarrow P^A_{-1} \stackrel{f^A_{-1}}\longrightarrow K^A_0  \stackrel{\delta^A_{-1}}\dashrightarrow \text{\quad and\quad } C\stackrel{g^C_{-1}}\longrightarrow P^C_{-1} \stackrel{f^C_{-1}}\longrightarrow K^C_0  \stackrel{\delta^C_{-1}}\dashrightarrow
\]
in $\xi$ with $P^A_{-1},P^C_{-1}\in\Px$ and $K^A_0,K^C_0\in\Gpx$ by definition. Since $C$ is in $\Gpx$, the $\Mbe$-triangle $A\stackrel{x}\> B \stackrel{y}\> C  \stackrel{\delta}\dashrightarrow$ is $\Mcc(-,\Mcx)$-exact by  Lemma \ref{HHH}(1). So there is the following commutative diagram

\[
	\xymatrix{
	A\ar[r]^{x}\ar[d]_{g_{-1}^A}&B\ar[r]^{y}\ar[d]^{g_{-1}^B}&C\ar@{-->}[r]^{\delta}\ar[d]^{g_{-1}^C}&\\
	P_{-1}^A\ar[r]^{\begin{tiny}\begin{bmatrix}
	1 \\
	0
	\end{bmatrix}\end{tiny}}\ar[d]_{f_{-1}^A}&P^B_{-1}\ar[r]^{\begin{tiny}\begin{bmatrix}
	0&1
	\end{bmatrix}\end{tiny}}\ar[d]^{f_{-1}^B}&P_{-1}^C\ar@{-->}^0[r]\ar[d]^{f_{-1}^C}&\\
	K^A_0\ar[r]^x\ar@{-->}[d]^{\delta_{-1}^A}& K^B_0\ar[r]^y\ar@{-->}[d]^{\delta_{-1}^B}& K^C_0\ar@{-->}[r]^{\delta}\ar@{-->}[d]^{\delta_{-1}^C}&\\
	&&&
	} 
\]
where all rows and columns are $\Mcc(-,\Mcx)$-exact $\mathbb{E}$-triangles in $\xi$ with $P^B_{-1}=:P_{-1}^A\oplus P_{-1}^C$ by \citep[Lemma 5]{JDPPR}. Since $K^A_0,K^C_0$ belong to $\Gpx$, by repeating this process, we can obtain a $\Mcc(-,\Mcx)$-exact complex 
\[
\xymatrix{
	 B\ar[r]&P^B_{-1}\ar[r]&P^B_{-2}\ar[r]&P^B_{-3}\ar[r]&\cdots
}
\]
Similarly, we can obtain a $\Mcc(-,\Mcx)$-exact complex 
\[
\xymatrix{
	\cdots\ar[r]&P^B_2\ar[r]&P^B_1\ar[r]&P^B_0\ar[r]&B.
}
\]
By pasting these  $\Mcc(-,\Mcx)$-exact complexes together, we obtain the follows $\Mcc(-,\Mcx)$-exact complex 
\[
\xymatrix{
	\cdots\ar[r]&P^B_2\ar[r]&P^B_1\ar[r]&P^B_0\ar[r]&P^B_{-1}\ar[r]&P^B_{-2}\ar[r]&P^B_{-3}\ar[r]&\cdots
}
\]
which implies $B$ is in $\Gpx$.

If $B$  is in $\Gpx$,  there is a $\Mcc(-, \Mcx)$-exact $\mathbb{E}$-triangle
\[
	B\stackrel{g^B_{-1}}\longrightarrow P^B_{-1} \stackrel{f^B_{-1}}\longrightarrow K^B_0  \stackrel{\delta^B_{-1}}\dashrightarrow
\]
in $\xi$ with $P_{-1}^B \in \Px$ and $K_{0}^B \in \Gpx$. By Lemma \ref{Dec30}(1), there exists the following commutative diagram
\[
	\xymatrix{
		A\ar@{=}[d] \ar[r]^x & B  \ar[r]^{y}\ar[d]^{g^B_{-1}} & C \ar[d]^{g} \ar@{-->}[r]^{\delta} &  \\
		A\ar[r]^{g^A_{-1}} & P^B_{-1}\ar[d]^{f^B_{-1}} \ar[r]^{f^A_{-1}} & G\ar[d]^{f} \ar@{-->}[r]^{\delta^A_{-1}} &  \\
		& K^B_0\ar@{=}[r]\ar@{-->}[d]^{\delta^B_{-1}}& K^B_0\ar@{-->}[d]^{y_*\delta^B_{-1}} \\
		& & &  } 
\] 
where all rows and columns are both $\Mcc(-,\Mcx)$-exact  $\mathbb{E}$-triangles in $\xi$.  Since $G$ lies in $\Gpx$, there is a $\Mcc(-,\Mcx)$-exact complex
\[
	\xymatrix{
		 G\ar[r]&P^A_{-2}\ar[r]&P^A_{-3}\ar[r]&P^A_{-4}\ar[r]&\cdots
	}
\]
with $P^A_n\in \Px$ for any $n\geq 2$.  Hence we get a $\Mcc(-,\Mcx)$-exact $\xi$-exact complex
\[\xymatrix{
	A\ar[r]& P^B_{-1}\ar[r]&P^A_{-2}\ar[r]&P^A_{-3}\ar[r]&\cdots
	}
\]
with $P_{-1}^B \in \Px$ and $P_{-n}^A \in \Px$ for any $n \geqslant 2$.  Since  $C\in\Gpx$, there exists an  $\mathbb{E}$-triangle $K_1^C\stackrel{g^C_0}\longrightarrow P^C_0 \stackrel{f^C_0}\longrightarrow C  \stackrel{\delta^C_0}\dashrightarrow$ in $\xi$ with  $P_0^C \in \Px$ and $K_1^C \in\Gpx$. We have the following commutative diagram 
	 \[
	 \xymatrix{
		 &A\ar@{=}[r]\ar[d]&A\ar[d]\\
		 K_1^C\ar[r]\ar@{=}[d]&G\ar[r]\ar[d]&B\ar[d]\ar@{-->}[r]&\\
		 K_1^C\ar[r]&P_0^C\ar[r]\ar@{-->}[d]&C\ar@{-->}[r]\ar@{-->}[d]&\\
		 &&&
	 }
\]	
made of $\mathbb{E}$-triangles in $\xi$ by Lemma \ref{BH}(1) and Lemma \ref{Thm3.2}. So we have $G\in\Gpx$ by the above proof. Then there exists an $\Mbe$-triangle $K_1^A\>P\>G\dashrightarrow$ with $K_1^A\in\Gpx$ and $P\in\Px$ in $\xi$.
 It follows from Lemma \ref{Thm3.2} and $\rm(ET4)^{op}$ that we have the following commutative diagram 
\[
		\xymatrix{
			K_1^A\ar@{=}[d] \ar[r]^{g^A_0} & P_0^A \ar[r]^{f^A_0}\ar[d] & A \ar[d] \ar@{-->}[r] &  \\
			K_1^A\ar[r] & P\ar[d] \ar[r] & G\ar[d] \ar@{-->}[r]& \\
			& P_0^C\ar@{=}[r]\ar@{-->}[d]&P_0^C\ar@{-->}[d] \\
			& & &  }
\] 
where all rows and columns are  $\mathbb{E}$-triangles in $\xi$. Note that $P$ and $P_0^C$ are in $\Px$, so we conclude that $P_0^A\in\Px$. Because of the  $\mathbb{E}$-triangle $A\stackrel{x}\longrightarrow B \stackrel{y}\longrightarrow C  \stackrel{\delta}\dashrightarrow$ is   $\Mcc(\P,-)$-exact, there  exists a morphism $a\in\Mcc(P^C_0,B)$ such that $f_0^C=ya$. So
\[
(f_0^C)^*\delta=(ya)^*0=a^*(y^*\delta)=0=(f_0^A)_*0
\]
by \cite[Corollary 3.5]{HY} and there is a $\xi$-deflation $f^B_0:P_0^A\oplus P_0^C=:P_0^B\to B $ by \citep[Proposition 1]{JDPPR}, which makes the following diagram 
\[
	\xymatrix{P_0^A\ar[r]^{\begin{tiny}\begin{bmatrix} 
	1 \\
	0
	\end{bmatrix}\end{tiny}}\ar[d]_{f^A_0}&P_0^B\ar[r]^{\begin{tiny}\begin{bmatrix}
	0&1
	\end{bmatrix}\end{tiny}}\ar@{-->}[d]^{f^B_0}&P_0^C\ar@{-->}^0[r]\ar[d]^{f^C_0}&\\
	A\ar[r]^x&B\ar[r]^y&C\ar@{-->}[r]^{\delta}&
	}
	\]
commutative. Thus, there is an $\Mbe$-triangle $K_1^B\stackrel{g^B_0}\longrightarrow P^B_0 \stackrel{f^b_0}\longrightarrow B  \stackrel{\delta^B_0}\dashrightarrow$ in $\xi$, which is $\Mcc(-,\Mcx)$-exact by  Lemma \ref{HHH}(1). Then we have the following commutative diagram 
\[
	\xymatrix{
	K_1^A\ar[r]^{x_1}\ar[d]_{g^A_0}&K_1^B\ar[r]^{y_1}\ar[d]^{g_0^B}&K_1^C\ar@{-->}[r]^{\delta_1}\ar[d]^{g_0^C}&\\
	P_0^A\ar[r]^{\begin{tiny}\begin{bmatrix}
	1 \\
	0
	\end{bmatrix}\end{tiny}}\ar[d]_{f_0^A}&P_0^B\ar[r]^{\begin{tiny}\begin{bmatrix}
	0&1
	\end{bmatrix}\end{tiny}}\ar[d]^{f_0^B}&P_0^C\ar@{-->}^0[r]\ar[d]^{f_0^C}&\\
	A\ar[r]^x\ar@{-->}[d]^{\delta_0^A}&B\ar[r]^y\ar@{-->}[d]^{\delta_0^B}&C\ar@{-->}[r]^{\delta}\ar@{-->}[d]^{\delta_0^C}&\\
	&&&
	} 
\]
where all rows and columns are $\mathbb{E}$-triangles in $\xi$ by \citep[Lemma 4.14]{JDP}. It is easy to show that the first vertical is $\Mcc(-,\Mcx)$ by $3\times 3$-Lemma. Recall that  $B$  is in $\Gpx$, so there is an $\mathbb{E}$-triangles 
$K'^B_0\stackrel{g'^B_{0}}\longrightarrow P'^B_{0} \stackrel{f'^B_{0}}\longrightarrow B\stackrel{\delta'^B_{0}}\dashrightarrow $
in $\xi$ by definition, where $K'^B_0\in\Gpx$ and $P'^B_{0}\in\Px$. Then we have $K^B_0\oplus P'^B_{0}\backsimeq K'^B_0\oplus P^B_{0}\in\Gpx$ by \cite[Proposition 4.3]{JDP}. Hence, any $\mathbb{E}$-triangle $K_2^B\stackrel{g^B_1}\longrightarrow P^B_1 \stackrel{f^B_1}\longrightarrow K^B_1  \stackrel{\delta^B_1}\dashrightarrow$ in $\xi$ is $\Mcc(-,\Mcx)$-exact by Lemma \ref{HHH} (2). Repeating this process, we can obtain a $\Mcc(-,\Mcx)$-exact complex 
\[
	\xymatrix{
		\cdots\ar[r]&P^A_2\ar[r]&P^A_1\ar[r]&P^A_0\ar[r]&A
	}
\]
with $P^A_n\in\Px$ for any $n\geq 0$. Therefore, we obtain that  $A$ is in $\Gpx$, as desired.
\end{proof}

The following result is crucial for this section.

\begin{thm}\label{Jan1}
	$\Gpx$ is closed under direct summands.
\end{thm}

\begin{proof}
	Assume that $G\in \Gpx$ and $M$ is a direct summand of $G$. Then there exists $M'\in\Mcc$ such that $G=M\oplus M'$. Therefore, there exist two split $\mathbb{E}$-triangles
\[
M\stackrel{\begin{tiny}\begin{bmatrix}
	1 \\
	0
	\end{bmatrix}\end{tiny}}\longrightarrow G \stackrel{\begin{tiny}\begin{bmatrix}
		0&1
		\end{bmatrix}\end{tiny}}\longrightarrow M'\stackrel{0}{\dashrightarrow} \text{\ and\ }
		M'\stackrel{\begin{tiny}\begin{bmatrix}
			0 \\
			1
			\end{bmatrix}\end{tiny}}\longrightarrow G \stackrel{\begin{tiny}\begin{bmatrix}
				1&0
				\end{bmatrix}\end{tiny}}\longrightarrow M\stackrel{0}{\dashrightarrow} 
\]
in $\xi$. Since $G\in\Gpx$, there exists a  $\Mcc(-,\Mcx)$-exact $\mathbb{E}$-triangle $G\stackrel{g_{-1}}\longrightarrow P_{-1} \stackrel{f_{-1}}\longrightarrow K_{0} \stackrel{\delta_{-1}}\dashrightarrow$ in $\xi$ with $P_{-1}\in\Px$ and $K_0\in\Gpx$. By Lemma \ref{Dec30}(1), there exists the following commutative diagram
\[
		\xymatrix{
			M\ar@{=}[d] \ar[r]^{\begin{tiny}\begin{bmatrix}
				1 \\
				0
				\end{bmatrix}\end{tiny}} & G  \ar[r]^{\begin{tiny}\begin{bmatrix}
					0&1
					\end{bmatrix}\end{tiny}}\ar[d]^{g_{-1}} & M' \ar[d]^{x'_{-1}} \ar@{-->}[r]^{0} &  \\
			M\ar[r]^{x_{-1}} & P_{-1}\ar[d]^{f_{-1}} \ar[r]^{y_{-1}} & X\ar[d]^{y'_{-1}} \ar@{-->}[r]^{\theta_{-1}} & \\
			& K_0\ar@{=}[r]\ar@{-->}[d]^{\delta_{-1}}&K_0\ar@{-->}[d]^{\theta'_{-1}=\begin{tiny}\begin{bmatrix}
				0&1
				\end{bmatrix}\end{tiny}_*\delta_{-1}} \\
			& & &  }
	\] 
	where all rows and columns are $\Mcc(-,\Mcx)$-exact  $\mathbb{E}$-triangles in $\xi$. Note that \[
M'\stackrel{x'_{-1}}\longrightarrow X \stackrel{y'_{-1}}\longrightarrow K_0\stackrel{\theta'_{-1}}{\dashrightarrow} \text{\ and\ }
		M'\stackrel{\begin{tiny}\begin{bmatrix}
			0 \\
			1
			\end{bmatrix}\end{tiny}}\longrightarrow G \stackrel{\begin{tiny}\begin{bmatrix}
				1&0
				\end{bmatrix}\end{tiny}}\longrightarrow M\stackrel{0}{\dashrightarrow} 
\]
are $\mathbb{E}$-triangles in $\xi$.  It follows Lemma \ref{BH}(2) and Lemma \ref{Thm3.2} that  we have  the following commutative diagram
 \[
	\xymatrix{
	           M'\ar[d]_{\begin{tiny}\begin{bmatrix}
					0 \\
					1
					\end{bmatrix}\end{tiny}}\ar[r]^{x'_{-1}}&X\ar[d]^{g'_{-1}}\ar[r]^{y'_{-1}}&K_0\ar@{=}[d]\ar@{-->}[r]^{\theta'_{-1}} &\\
				G\ar[d]_{\begin{tiny}\begin{bmatrix}
					1&0
					\end{bmatrix}\end{tiny}}\ar[r]^{x''_{-1}}&G_{-1}\ar[r]^{y''_{-1}}\ar[d]^{f'_{-1}}&K_0\ar@{-->}[r]^{\theta''_{-1}} &\\
				M\ar@{=}[r]\ar@{-->}[d]^{0}&M\ar@{-->}[d]^{0}\\
				&&
			}
\]
where \[
G\stackrel{x''_{-1}}\longrightarrow G_{-1} \stackrel{y''_{-1}}\longrightarrow K_0\stackrel{\theta''_{-1}}{\dashrightarrow} \text{\ and\ }
		X\stackrel{g'_{-1}}\longrightarrow G_{-1} \stackrel{f'_{-1}}\longrightarrow M\stackrel{0}{\dashrightarrow} 
\]
are  $\mathbb{E}$-triangles in $\xi$. Moreover, the  two $\mathbb{E}$-triangles above are $\Mcc(-,\Mcx)$-exact by Lemma \ref{HHH}. Because $G$ and $K_0$ are in $\Gpx$, the object $G_{-1}$ belongs to $\Gpx$ by Theorem \ref{Thm5.8}. Therefore, there is  a  $\Mcc(-,\Mcx)$-exact $\mathbb{E}$-triangle $G_{-1}\stackrel{g_{-2}}\longrightarrow P_{-2} \stackrel{f_{-2}}\longrightarrow K_{-1} \stackrel{\delta_{-2}}\dashrightarrow$ in $\xi$ with $P_{-2}\in\Px$ and $K_{-1}\in\Gpx$. So we have the following commutative diagram
\[
		\xymatrix{
			X\ar@{=}[d] \ar[r]^{g'_{-1}} & G_{-1}  \ar[r]^{f'_{-1}}\ar[d]^{g_{-2}} & M \ar[d]^{x'_{-2}} \ar@{-->}[r]^{0} &  \\
			X\ar[r]^{x_{-2}} & P_{-2}\ar[d]^{f_{-2}} \ar[r]^{y_{-2}} & Y\ar[d]^{y'_{-2}} \ar@{-->}[r]^{\theta_{-2}} & \\
			& K_{-1}\ar@{=}[r]\ar@{-->}[d]^{\delta_{-2}}&K_{-1}\ar@{-->}[d]^{(f'_{-1})_*\delta_{-2}} \\
			& & &  }
	\] 
where all rows and columns are  $\Mcc(-,\Mcx)$-exact  $\mathbb{E}$-triangles in $\xi$ by Lemma \ref{Dec30}(1). Proceedings this manner, we can obtain a $\Mcc(-,\Mcx)$-exact complex 
\[\xymatrix{
	M\ar[r]& P_{-1}\ar[r]&P_{-2}\ar[r]&P_{-3}\ar[r]&\cdots
	}
\]
with $P_n\in\Px$ for any $n<0$. Similarly, we can get the following  $\Mcc(-,\Mcx)$-exact complex
\[
	\xymatrix{
		\cdots\ar[r]&P_2\ar[r]&P_1\ar[r]&P_0\ar[r]&M
	}
\]
with $P_n\in\Px$ for any $n\geq 0$. Hence, $M\in\Gpx$, as desired.
\end{proof}
 
\begin{prop}\label{Prop5.10}
Let $M$ be an object in $\Mcc$. Assume that there is a $\xi$-exact complex 
\[
\cdots\>G_n\>\cdots\>G_1\>G_0\>M\>0 \tag{5.1}
\]
with all $G_n\in\Gpx$. Then there is a $\xi$-exact complex 
\[
\cdots\>P_n\>\cdots\>P_1\>P_0\>M\>0 \tag{5.2}
\]
with all $P_n\in\Px$. Moreover, if the $\xi$-exact complex (5.1) is $\Mcc(-,\Mcx)$-exact, then so is (5.2).
\end{prop}
\begin{proof}
From (5.1), we have the following $\Mbe$-triangles 
\[
K_1\>G_0\>M\dashrightarrow\text{\ and \ }  K_2\>G_1\>K_1\dashrightarrow
\]
in $\xi$. There exists an $\Mbe$-triangle $G\>P_0\>G_0\dashrightarrow$ in $\xi$ with $P_0\in\Px$ and $G\in\Gpx$, since $G_0\in\Gpx$. By Lemma \ref{Thm3.2} and $\rm(ET4)^{op}$, we have the following commutative diagram 
\[
		\xymatrix{
			G\ar@{=}[d] \ar[r] & L \ar[r]\ar[d] & K_1 \ar[d] \ar@{-->}[r] &  \\
			G\ar[r] & P_0\ar[d] \ar[r] & G_0\ar[d] \ar@{-->}[r]& \\
			& M\ar@{=}[r]\ar@{-->}[d]&M\ar@{-->}[d] \\
			& & &  }\tag {$*$}
\] 
where all  rows and columns are  $\mathbb{E}$-triangles in $\xi$. Thus, we have the following commutative diagram 
\[
\xymatrix{
	&G\ar@{=}[r]\ar[d]&G\ar[d]\\
	K_2\ar[r]\ar@{=}[d]&G'\ar[r]\ar[d]&L\ar[d]\ar@{-->}[r]&\\
	K_2\ar[r]&G_1\ar[r]\ar@{-->}[d]&K_1\ar@{-->}[r]\ar@{-->}[d]&\\
	&&&
}\tag{$**$}
\]	
made of $\mathbb{E}$-triangles in $\xi$ by Lemma \ref{BH}(1) and Lemma \ref{Thm3.2}. So $G'\in\Gpx$ by Theorem \ref{Thm5.8}. Repeating this process, we can obtain a $\xi$-exact complex 
\[
\cdots\>P_n\>\cdots\>P_1\>P_0\>M\>0 
\]
with all $P_n\in\Px$. Moreover, if (5.1) is $\Mcc(-,\Mcx)$-exact,  we can conclude that the $\Mbe$-triangle $G\>L\>K_1\dashrightarrow$ is $\Mcc(-,\Mcx)$-exact by Lemma \ref{Dec30} from diagram $(*)$. We fix $X\in\Mcx$. Applying $\Mcc(-,X)$ to the diagram ($**$), one obtains the following commutative diagram of abelian groups:
\[
	\xymatrix{
		&0\ar[d]&0\ar[d]&0\ar[d]\\
		0\ar[r]&\mathcal{C}(K_1,X)\ar[d]\ar[r]&\mathcal{C}(G_1,X)\ar[d]\ar[r] & \Mcc(K_2,X)\ar@{=}[d]\ar[r]&0\\
		0\ar@{..>}[r]&\mathcal{C}(L,X) \ar[d]\ar[r]& \mathcal{C}(G',X)  \ar[r]\ar[d]& \mathcal{C}(K_2,X) \ar[d] \ar@{..>}[r]&0  \\
		0\ar[r]&\mathcal{C}(G,X)\ar[d]\ar@{=}[r] & \mathcal{C}(G,X)\ar[d] \ar[r] & 0 \ar[d] \ar[r]&0\\
		&0&0&0
			}
\]
So the $\Mbe$-triangle $K_2\>G'\>L\dashrightarrow$ is $\Mcc(-,X)$-exact by $3\times 3$-Lemma. Repeating this process, we obtain that the $\xi$-exact complex (5.2) is $\Mcc(-,\Mcx)$-exact.
\end{proof}

The next result is dual to Proposition \ref{Prop5.10}.

\begin{prop}
Let $M$ be an object in $\Mcc$. Assume that there is a $\xi$-exact complex 
\[
0\>M\>G_{-1}\>G_{-2}\>\cdots\>G_{-n}\>\cdots\tag{5.3}
\]
with all $G_{-n}\in\Gpx$. Then there is a $\xi$-exact complex 
\[
0\>M\>P_{-1}\>P_{-2}\>\cdots\>P_{-n}\>\cdots\tag{5.4}
\]
with all $P_{-n}\in\Px$. Moreover, if the $\xi$-exact complex (5.3) is $\Mcc(-,\Mcx)$-exact, then so is (5.4).
\end{prop}

\begin{cor}\label{Cor5.12}
	Let $M$ be an object in $\Mcc$, then $M$ admits a  $\xi$-exact complex
\[
\cdots\>G_n\>\cdots\>G_1\>G_0\>M\>0 
\]
with all $G_n\in\Gpx$  which is $\Mcc(-,\Mcx)$-exact if and only if $M$ admits a  $\xi$-exact complex
\[
\cdots\>P_n\>\cdots\>P_1\>P_0\>M\>0 
\]
with all $P_n\in\Px$  which is $\Mcc(-,\Mcx)$-exact.
\end{cor}
Dually, we have the following corollary.
\begin{cor}\label{Cor5.13}
	Let $M$ be an object in $\Mcc$, then $M$ admits a $\xi$-exact complex
\[
0\>M\>G_{-1}\>G_{-2}\>\cdots\>G_{-n}\>\cdots
\]
with all $G_{-n}\in\Gpx$ which is $\Mcc(-,\Mcx)$-exact if and only if $M$ admits a  $\xi$-exact complex
\[
	0\>M\>P_{-1}\>P_{-2}\>\cdots\>P_{-n}\>\cdots
\]
with all $P_{-n}\in\Px$ which is $\Mcc(-,\Mcx)$-exact.
\end{cor}

\begin{lem}
	Let $M$ be an object in $\Mcc$, then the following statemants are equivalent:

(1) $M$ is a $\mathcal{GQ_{\Mcx}(\xi)}$-projective object.

(2) There exist two $\xi$-exact complexes 
\[
\cdots\>P_1\>P_0\>M\>0 \text{\ and \ } 0\>M\>P_{-1}\>P_{-2}\>\cdots
\]
with $P_n\in\Px$ for any $n\in\mathbb{Z}$, which are both $\Mcc(-,\Mcx)$-exact.

(3) There exist $\Mcc(-,\Mcx)$-exact $\Mbe$-triangles $K_{n+1}\>P_n\>K_n\dashrightarrow$ in $\xi$ such that $P_n\in\Px$ for any $n\in\mathbb{Z}$ and $K_0=M$.
\end{lem}
\begin{proof}
It follows from the definition of $\mathcal{GQ_{\Mcx}(\xi)}$-projective object.
\end{proof}

\begin{thm}\label{Thm5.15}
	Let $M$ be an object in $\Mcc$, then the following statemants are equivalent:

(1) $M$ is a $\mathcal{GQ_{\Mcx}(\xi)}$-projective object.

(2)  There exist $\Mcc(-,\Mcx)$-exact and $\Mcc(\Mcx,-)$-exact  $\Mbe$-triangles $K_{n+1}\>G_n\>K_n\dashrightarrow$ in $\xi$ such that $G_n\in\Gpx$ for any $n\in\mathbb{Z}$ and $K_0=M$.

(3) There exist $\Mcc(-,\Mcx)$-exact $\Mbe$-triangles $K_{n+1}\>G_n\>K_n\dashrightarrow$ in $\xi$ such that $G_n\in\Gpx$ for any $n\in\mathbb{Z}$ and $K_0=M$.

(4) There exists a $\Mcc(-,\Mcx)$-exact and $\Mcc(\Mcx,-)$-exact  $\Mbe$-triangle $M\>G\>M\dashrightarrow$ in $\xi$ with $G\in\Gpx$.

(5) There exists a $\Mcc(-,\Mcx)$-exact $\Mbe$-triangle $M\>G\>M\dashrightarrow$ in $\xi$ with $G\in\Gpx$.
\end{thm}
\begin{proof}
(1) $\Rightarrow$ (2) Let $M$ be a $\mathcal{GQ_{\Mcx}(\xi)}$-projective object of $\Mcc$. Consider the following split $\Mbe$-triangles
\[
0\>M\stackrel{1}\> M\dashrightarrow\text{\ and \ } M\stackrel{1}\>M\> 0\dashrightarrow
\]
in $\xi$, which is  $\Mcc(-,\Mcx)$-exact and $\Mcc(\Mcx,-)$-exact.

(2) $\Rightarrow$ (3) It is clear.

(1) $\Rightarrow$ (4)  Since $M\in\Gpx$, we have $M\oplus M\in\Gpx$ by Theorem \ref{Thm5.8}. Consider the following split $\Mbe$-triangle
\[
	M \stackrel{\begin{tiny}\begin{bmatrix}
		1 \\
		0
		\end{bmatrix}\end{tiny}}\>M \oplus M \stackrel{\begin{tiny}\begin{bmatrix}
		0&1
		\end{bmatrix}\end{tiny}}\> M\dashrightarrow
\]
in $\xi$, which is  $\Mcc(-,\Mcx)$-exact and $\Mcc(\Mcx,-)$-exact.

(4) $\Rightarrow$ (5) $\Rightarrow$ (3)  It is clear.

(3) $\Rightarrow$ (1) It follows from Corollary \ref{Cor5.12} and Corollary \ref{Cor5.13}.
\end{proof}

Let $\left[\Gpx\right]^1=\Gpx$, and inductively set the following subcategory of $\Mcc$:
\[
\begin{split}
\left[\Gpx\right]^{n+1}=&\left\{M \in \Mcc \mid\text{there exists a $\Mcc(-,\Mcx)$-exact and $\xi$-exact complex}\right. \\
&\cdots \> G_1 \> G_0 \> G_{-1} \> G_{-2} \>\cdots \text{\ in\ }\Mcc\\
&\text{with all\ } G_i\in\Gpx \text{\ and a $\xi$-resolution $\Mbe$-triangle\ } K_1\>G_0\>M\dashrightarrow\left.\right\}.
\end{split}
\]

\begin{thm}
	For any $n\geq 1$, $\left[\Gpx\right]^n=\Gpx$.
\end{thm}
\begin{proof}
	It is easy to see that 
\[
\Gpx\subseteq \left[\Gpx\right]^2\subseteq \left[\Gpx\right]^3 \subseteq\cdots
\]
is an ascending chain of subcategories of $\Mcc$. By (1) $\Leftrightarrow$ (3) of the Theorem \ref{Thm5.15}, we have that $\left[\Gpx\right]^2=\Gpx$. By using induction on $n$, we obtain easily the assertion
\end{proof}

We give some characterizations of $\Gpx$-resolution dimensions in the following part.

\begin{prop}\label{Prop5.17}
	Let $A\longrightarrow B\longrightarrow C \dashrightarrow$ be an $\mathbb{E}$-triangle in $\xi$ with $A,B\in\Gpx$, then the following statemants are equivalent:

(1) $C\in\Gpx$.

(2) $\ext^1(C,X)=0$, $\forall X\in\Mcx$.

(3) $\ext^1(C,P)=0$, $\forall P\in\Px$.
\end{prop}

\begin{proof}
(1) $\Rightarrow$ (2) It follows from Lemma \ref{TTTT}.

(2) $\Rightarrow$ (3) It is trivial, since $\Px\subseteq\Mcx$.

(3) $\Rightarrow$ (1) Since $A\in\Gpx$, there is an $\mathbb{E}$-triangle $A\longrightarrow P \longrightarrow K\dashrightarrow$  with $P\in\Px$ and $K\in\Gpx$. Then we have the following commutative diagram
	\[
		\xymatrix{
			A\ar[d]\ar[r]&B\ar[d]\ar[r]&C\ar@{=}[d]\ar@{-->}[r] &\\
			P\ar[d]\ar[r]&G\ar[r]\ar[d]&C\ar@{-->}[r] &\\
			K\ar@{=}[r]\ar@{-->}[d]&K\ar@{-->}[d]\\
			&&}
\]
where all rows and columns are $\mathbb{E}$-triangles in $\xi$ by Lemma \ref{BH}(2) and Lemma \ref{Thm3.2}. So we obtain $G\in\Gpx$. For the $\mathbb{E}$-triangle $P\longrightarrow G \longrightarrow C\dashrightarrow$, we have the following commutative diagram by hypothesis
	\[
	\xymatrix{
	 &\Mcc(C,P)\ar[r]\ar[d]&\Mcc(G,P)\ar[r]\ar[d]^{\varphi}_{\simeq }&\Mcc(P,P)\ar@{..>}[r]\ar[d]_{\simeq }^{\psi } & 0&\\
	0\ar[r]&\ext^0(C,P)\ar[r]&\ext^0(G,P)\ar[r]&\ext^0(P,P)\ar[r]&\ext^1(C,P)=0.
	}
\]
where $\varphi$ and $\psi$ are isomorphisms by Lemma \ref{TTTT} since $P\in\Px\subseteq \Gpx$. Hence, the $\mathbb{E}$-triangle $P\longrightarrow G \longrightarrow C\dashrightarrow$ is split and then $C\in\Gpx$ by Theorem \ref{Jan1}.
\end{proof}

\begin{thm}\label{thm5.19}
Let $M$ be an object in $\Mcc$ with $\resdim{\Gpx}{M}<\infty$, then the following statemants are equivalent for any  $n\geq 0$:

(1) $\resdim{\Gpx}{M}\leq n$.

(2) There exists an $\Mbe$-triangle 
\[
M\>W\>G\dashrightarrow 
\]
in $\xi$ with $G\in\Gpx$ and $\resdim{\Px}{W}\leq n$.

(3) There exists an $\Mbe$-triangle 
\[
K\>G'\>M\dashrightarrow 
\]
in $\xi$ with $G'\in\Gpx$ and $\resdim{\Px}{K}\leq n-1$.

(4) There exist two $\Mbe$-triangles 
\[
W_M\>G_M\>M\dashrightarrow \text{\  and \ } M\>W^M\>G^M\dashrightarrow 
\]
in $\xi$ such that  $G_M, G^M$ belong to $\Gpx$ and $\resdim{\Px}{W_M}\leq n-1$, $\resdim{\Px}{W^M}=\resdim{\Gpx}{W^M}\leq n$. 

(5) There is a $\xi$-exact complex 
\[
0\>P_n\>P_{n-1}\>\cdots\> P_1\> G\>M\>0
\]
with $G\in\Gpx$ and all $P_i\in\Px$.

(6) For any $0\leq m\leq n$, there exists a $\xi$-exact complex 
\[
0\>P_n\>\cdots\> P_{m+1}\>G_m\>P_{m-1}\>\cdots\> P_1\> P_0\>M\>0
\]
with $G_m\in\Gpx$ and all $P_i\in\Px$.

(7) $\ext^{n+i}(M,P)=0$, $\forall P\in\Px$, $\forall i\geq 1$.

(8) $\ext^{n+i}(M,Q)=0$, $\forall Q\in\widehat{\Px}$, $\forall i\geq 1$.

(9) $\ext^{n+i}(M,X)=0$, $\forall X\in\Mcx$, $\forall i\geq 1$.

(10) $\ext^{n+i}(M,N)=0$, $\forall N\in\widehat{\Mcx}$, $\forall i\geq 1$.
\end{thm}

\begin{proof}
 (1) $\Leftrightarrow$ (2) $\Leftrightarrow$ (3)  $\Leftrightarrow$ (4) $\Leftrightarrow$ (5) $\Leftrightarrow$ (6) $\Leftrightarrow$ (7)  $\Leftrightarrow$ (8)  It follows from Theorem \ref{Thm4.6} since $\Px$ is a $\Ext$-injective $\xi$-cogenerator of $\Gpx$.

(1) $\Rightarrow$ (9) Since $\resdim{\Gpx}{M}\leq n$, then there exists a $\xi$-exact complex
\[
0\>G_n\>G_{n-1}\>\cdots\>G_1\>G_0\>M\>0
\]
with all $G_i\in\Mcx$ and the $\xi$-resolution $\mathbb{E}$-triangle $K_{i+1}\>G_i\>K_i\dashrightarrow$ (set $K_0=M$ and $K_n=G_n$) for any  $0\leq i\leq n$. By Lemma \ref{LZHL} and dimension shifting, we have $\ext^{n+i}(M,X)\cong \ext^i(G_n,X)=0$ for any $X\in\Mcx$ and $i\geq 1$ by Lemma \ref{TTTT}.

(9) $\Rightarrow$ (1)  Since $\resdim{\Gpx}{M}<\infty$, there exists a $\xi$-exact complex
\[
		\xymatrix{
			G_m\ar[r]&G_{m-1}\ar[r]&\cdots\ar[r]&G_2\ar[r]&G_1\ar[r]&G_0\ar[r]&M
	}
	\] 
for some integer $m$, where $G_i\in\Gpx$  and the $\xi$-resolution $\Mbe$-triangles are $K_{i+1}\longrightarrow G_i \longrightarrow K_i\dashrightarrow$  (set $K_0=M, K_m=G_m$) for any  $0\leq i\leq n-1$.  We only need to consider $m>n$. By dimension shifting, one can obtain that $\ext^i(K_j,X)\cong \ext^{i+j}(M,X)=0$ for any $i\geq 1$, $j\geq n$ and $X\in\Mcx$. Applying Proposition \ref{Prop5.17} on the $\Mbe$-triangle $G_m\rightarrow G_{m-1}\rightarrow K_{m-1}\dashrightarrow$ in $\xi$, implies that $K_{m-1}$ lies in $\Gpx$. By repeating this on the other $\xi$-resolution $\Mbe$-triangles, one can deduce that $K_n\in\Gpx$, and so $\resdim{\Gpx}{M}\leq n$.

(9) $\Rightarrow$ (10) Let $\resdim{\Mcx}{N}=m$, then then there exists a $\xi$-exact complex
\[
0\>X_m\>X_{m-1}\>\cdots\>X_1\>X_0\>N\>0
\]
with all $X_i\in\Mcx$ and the $\xi$-resolution $\mathbb{E}$-triangle $K_{i+1}\>G_i\>K_i\dashrightarrow$ (set $K_0=N$ and $K_m=X_m$) for any  $0\leq i\leq m$.  By Lemma \ref{LZHL} and dimension shifting, we have $\ext^{n+i}(M,N)\cong \ext^{m+n+i}(M,X_m)=0$ for any $i\geq 1$.

(10) $\Rightarrow$ (9) It is clear.
\end{proof}

From Theorem \ref{thm5.19}, we have the following consequence.

\begin{cor}
Let $M$ be an object in $\Mcc$ with $\resdim{\Gpx}{M}<\infty$, then
\[
\begin{split}
\resdim{\Gpx}{M}&=\sup\left\{n\in\mathbb{N} \mid \exists P\in\Px \text{\ such that\ } \ext^n(M,P)\neq 0 \right\}\\
&=\sup\left\{n\in\mathbb{N} \mid \exists Q\in\widehat{\Px} \text{\ such that\ } \ext^n(M,Q)\neq 0 \right\}\\
&=\sup\left\{n\in\mathbb{N} \mid \exists X\in\Mcx \text{\ such that\ } \ext^n(M,X)\neq 0 \right\}\\
&=\sup\left\{n\in\mathbb{N} \mid \exists N\in\widehat{\Mcx} \text{\ such that\ } \ext^n(M,N)\neq 0 \right\}.\\
\end{split}
\]
\end{cor}

\begin{cor}
	Let $M$ be an object in $\Mcc$ with $\resdim{\Px}{M}<\infty$, then
\[\begin{split}
\resdim{\Px}{M}&=\sup\left\{n\in\mathbb{N} \mid \exists P\in\Px \text{\ such that\ } \ext^n(M,P)\neq 0 \right\}\\
&=\sup\left\{n\in\mathbb{N} \mid \exists Q\in\widehat{\Px} \text{\ such that\ } \ext^n(M,Q)\neq 0 \right\}
\end{split}\]
\end{cor}
\begin{proof}
It follows that $\Px$ is a quasi-resolving subcategory of $\Mcc$ that serves as its own $\Ext$-injective $\xi$-cogenerator. 
\end{proof}

\begin{cor}
		Let $M$ be an object in $\Mcc$ with $\resdim{\Px}{M}<\infty$, then
\[
\resdim{\Px}{M}=\resdim{\Gpx}{M}.
\]
\end{cor}
{\small

}

\end{document}